\documentclass[draft, 12pt]{amsart}
\usepackage{amsmath,amsthm,amssymb}

\usepackage{dsfont}
\usepackage{mathrsfs} %use \mathscr
\usepackage[all,cmtip]{xy}
\usepackage{patchcmd}
\input{diagxy}
\usepackage{url}
\usepackage{bm}
\usepackage{mathtools}
\usepackage{enumitem}
\usepackage{stmaryrd}
\usepackage[usenames,dvipsnames,svgnames,table]{xcolor}
\usepackage[normalem]{ulem}
\usepackage[top=2.2cm, bottom=2.2cm, left=2cm, right=2cm]{geometry}
\usepackage{tikz}
\usetikzlibrary{arrows,positioning}
\allowdisplaybreaks

\makeatletter
\patchcommand\@starttoc{\begin{quote}}{\end{quote}}
\def\@tocline#1#2#3#4#5#6#7{\relax
  \ifnum #1>\c@tocdepth % then omit
  \else
    \par \addpenalty\@secpenalty\addvspace{#2}%
    \begingroup \hyphenpenalty\@M
    \@ifempty{#4}{%
      \@tempdima\csname r@tocindent\number#1\endcsname\relax
    }{%
      \@tempdima#4\relax
    }%
    \parindent\z@ \leftskip#3\relax \advance\leftskip\@tempdima\relax
    \rightskip\@pnumwidth plus4em \parfillskip-\@pnumwidth
    #5\leavevmode\hskip-\@tempdima
      \ifcase #1
       \or\or \hskip 1em \or \hskip 2em \else \hskip 3em \fi%
      #6\nobreak\relax
    \dotfill\hbox to\@pnumwidth{\@tocpagenum{#7}}\par
    \nobreak
    \endgroup
  \fi}
\makeatother

 \theoremstyle{plain}
 \newtheorem{thm}{Theorem}[section]
 \newtheorem{cor}[thm]{Corollary}
 \newtheorem{lem}[thm]{Lemma}

\theoremstyle{definition}
 \newtheorem{defn}[thm]{Definition}
 
 \newtheorem{hyp}[thm]{Hypothesis}

\theoremstyle{remark}
 \newtheorem{rem}[thm]{Remark}

 \newtheorem{ter}[thm]{Terminology}
 \newtheorem{nota}[thm]{Notation}
 
 \newtheorem{exam}[thm]{Example}

 \numberwithin{equation}{section}

\theoremstyle{plain}

\DeclareMathOperator{\VF}{VF}

\DeclareMathOperator{\RV}{RV}

\DeclareMathOperator{\MM}{\gM}
\DeclareMathOperator{\OO}{\gO}
\DeclareMathOperator{\UU}{\mathfrak{U}}

 \DeclareMathOperator{\id}{id}

 \DeclareMathOperator{\Th}{Th}

 \DeclareMathOperator{\pr}{pr}

\DeclareMathOperator{\K}{\mathds{k}}

\DeclareMathOperator{\res}{res}  % map into the residue field

\def\XXint#1#2#3{{\setbox0=\hbox{$#1{#2#3}{\int}$}
\vcenter{\hbox{$#2#3$}}\kern-.5\wd0}}

\newcommand{\KK}{\mathds{K}}
\newcommand{\Q}{\mathds{Q}}
\newcommand{\N}{\mathds{N}}

\newcommand{\R}{\mathds{R}}

\newcommand{\omin}{$o$\nobreakdash}

\newcommand{\T}{$T$\nobreakdash}

\newcommand{\dand}{\quad \text{and} \quad}

\newcommand{\gM}{\mathfrak{M}}

\newcommand{\gO}{\mathfrak{O}}

\newcommand{\ga}{\mathfrak{a}}
\newcommand{\gb}{\mathfrak{b}}

\newcommand{\gd}{\mathfrak{d}}

\newcommand{\gu}{\mathfrak{u}}

\newcommand{\0}{\emptyset}

\DeclareMathAlphabet{\mathpzc}{OT1}{pzc}{m}{it}

\providecommand\given{}
\newcommand\SetSymbol[1][]{%
\nonscript \: #1 \vert
\allowbreak
\nonscript\:
\mathopen{}}
\DeclarePairedDelimiterX\set[1]\{\}{%
\renewcommand\given{\SetSymbol[\delimsize]}
#1
}
 \DeclarePairedDelimiter\abs{\lvert}{\rvert}

 \newcommand{\lan}[3]{\mathcal{L}_{#1 \textup{#2} #3}}

\newcommand{\mdl}[1]{\mathcal{#1}}  % model; e.g. M
\newcommand{\rest}{\upharpoonright}

\newcommand{\fun}{\longrightarrow}
\newcommand{\efun}{\longmapsto}
\newcommand{\sub}{\subseteq}

\newcommand{\mi}{\smallsetminus}

\newbox\gnBoxA
\newdimen\gnCornerHgt
\setbox\gnBoxA=\hbox{$\ulcorner$}
\global\gnCornerHgt=\ht\gnBoxA
\newdimen\gnArgHgt

\def\code #1{%
        \setbox\gnBoxA=\hbox{$#1$}%
        \gnArgHgt=\ht\gnBoxA%
        \ifnum \gnArgHgt<\gnCornerHgt
                \gnArgHgt=0pt%
        \else
                \advance \gnArgHgt by -\gnCornerHgt%
        \fi
        \raise\gnArgHgt\hbox{$\ulcorner$} \box\gnBoxA %
                \raise\gnArgHgt\hbox{$\urcorner$}}

\DeclareMathOperator{\rv}{rv}

\DeclareMathOperator{\vv}{val}

\DeclareMathOperator{\rad}{rad}

\DeclareMathOperator{\RCF}{RCF}

\DeclareMathOperator{\TCVF}{TCVF}
\newcommand{\LT}{$\lan{T}{}{}$\nobreakdash}

\newcommand{\ddx}{\tfrac{d}{d x}}

\DeclareMathOperator{\abv}{\abs{\vv}}

\newcommand{\ticvf}[1]{\text{T}_{#1}\text{CVF}}
\DeclareMathOperator{\ssin}{\Yright}
\newcommand{\usu}{u \ssin \UU^+}
\newcommand{\cpt}{\complement}
\begin{document}

\title[Approximation in $\TCVF$]{Approximation by $O$-minimal sets in power-bounded $T$-convex valued fields}

\author[Y. Yin]{Yimu Yin}
\address{Pasadena, California}
\email{yimu.yin@hotmail.com}

\thanks{The research leading to the true claims in this paper has been partially supported by  the SYSU grant 11300-18821101}



\dedicatory{In memory of Jean}

\begin{abstract}
We show that, for a certain large class of power-bounded \omin-minimal $\mathcal{L}_T$-theories $T$ whose field of exponents is infinite-dimensional as a $\Q$-vector space,  any definable set in a \T-convex valued field $(\mathcal{R}, \mathfrak{O})$ is in a precise sense the limit of a family of $\mathcal{L}_T$-definable sets indexed over the residue field. Alternatively, in the mainstream model-theoretic language, this says that if  $(\mathcal{R}', \mathfrak{O}')$ is an elementary substructure of $(\mathcal{R}, \mathfrak{O})$ and if the residue field of $\mathfrak{O}$ contains an element that is infinitesimal relative to the residue field of $\mathfrak{O}'$ then any set $A \sub (\mathcal{R}')^m$ definable in $(\mathcal{R}', \mathfrak{O}')$ is the trace of a set definable in $\mdl R$.
\end{abstract}

\maketitle

\tableofcontents

\section{Introduction}

Let $T$ be a complete \omin-minimal \LT-theory extending the theory $\RCF$ of real closed fields and $\mdl R$ a \T-model. For ease of notation, we will not distinguish a structure from its underlying set  when speaking of definable sets in the former. Recall  that an elementary substructure $\mdl R'$ of $\mdl R$ is \emph{tame} in $\mdl R$ and, dually, $\mdl R$ is a \emph{tame extension} of $\mdl R'$  if, for every $\mdl R'$-bounded $r \in \mdl R$, there is a (necessarily unique) $r' \in \mdl R'$ such that $\abs{r - r'} < \epsilon$ for all positive $\epsilon \in \mdl R'$.

\begin{thm}[\cite{makr:stein:deftype}]\label{makrstein}
Suppose that $\mdl R'$ is tame in $\mdl R$. Then, for any definable set $A \sub \mdl R^m$, its trace $A \cap (\mdl R')^m$ in $\mdl R'$ is definable in $\mdl R'$.
\end{thm}

Since $T$ is assumed to be an extension of $\RCF$, not just any \omin-minimal theory, there is a proof of this result, utilizing the theory of \T-convexity as developed in  \cite{DriesLew95}, that is much shorter than the original one in \cite{makr:stein:deftype}; see \cite{tressl:marste}.

Recall that a nonempty proper convex subring $\OO$ of $\mdl R$ is \emph{\T-convex} if for every $\0$-definable continuous function $f : \mdl R \fun \mdl R$, we have $f(\OO) \sub \OO$. Note that the convexity of $\OO$ implies that $\OO$ is a valuation ring. More intuitively, if $\mdl R$ extends the real field $\R$ and $\OO$ is the convex hull of $\R$ in $\mdl R$ then  $\OO$ being \T-convex just means that  no $\0$-definable continuous function can grow so fast as to  stretch the standard reals into infinity. According to \cite{DriesLew95}, the theory $T_{\textup{convex}}$ of the pair $(\mdl R, \OO)$ with $\OO$ a \T-convex subring of $\mdl R$, suitably axiomatized in the language $\lan{}{convex}{}$ that extends $\lan{T}{}{}$ with a new unary relation symbol, is complete.

Let $(\mdl R, \OO) \models T_{\textup{convex}}$. Now, if $\mdl R'$ is not tame in $\mdl R$ then the trace in $\mdl R'$ of an \LT-definable set in $\mdl R$ may not be a definable set in $\mdl R'$. For instance, if $(\mdl R', \OO') \preccurlyeq (\mdl R, \OO)$ then $\OO'$ could be the trace of an interval. This suggests the question: Is it always the case that the trace in $\mdl R'$ of an \LT-definable set in $\mdl R$ a definable set in $(\mdl R', \OO')$? Unfortunately, it is easy to conceive  a counterexample as follows. Suppose that $T$ is polynomially bounded (or, more generally, power-bounded). Then every subset of $\mdl R'$ that is definable in $(\mdl R', \OO')$ is a boolean combination of intervals and valuative discs (see Remark~\ref{HNF}). It follows that if there is a descending sequence of valuative discs in $(\mdl R', \OO')$ whose intersection is empty in $\mdl R'$ but does contain a point $a \in \mdl R \mi \mdl R'$ then the trace in $\mdl R'$ of the interval $(0, a) \sub \mdl R$ cannot be   definable in $(\mdl R', \OO')$.

One may of course search a condition on the pair $(\mdl R', \mdl R)$, much like the tameness requirement in Theorem~\ref{makrstein}, that makes the trace-is-definable property holds. However, in this paper, our goal is to investigate  the dual question: Under what circumstances is every definable set in $(\mdl R', \OO')$ the trace of an \LT-definable set in $\mdl R$?

For simplicity, in the remainder of this introduction, we assume $T = \RCF$ and the structure $(\mdl R, \OO)$ is sufficiently saturated. Sets definable in $\mdl R$ are also called semialgebraic in the literature of real geometry and topology. If a  set $S$ is equipped with a total ordering and a distinguished element $e$ then it makes sense to speak of the positive and the negative parts of $S$ relative to $e$ (excluding $e$ itself), which we denote  by $S^+$ and $S^-$, respectively.

Write  $\R$ as $\bar \R$ to emphasize that it is a model of $\RCF$. By Wilkie's Theorem, the theory of the structure $\tilde{\mdl A}_0 = (\bar \R, \exp)$ is \omin-minimal. Let $\tilde{\mdl A}_i = (\mdl A_i, \exp)$, $i > 0$, be a sequence of ascending proper elementary extensions of $\tilde{\mdl A}_0$. Let $a_i \in \tilde{\mdl A}_i^+$ with $a_i > \tilde{\mdl A}_{i-1}$ (take $a_0 = 1$) and $\zeta_i : \tilde{\mdl A}_i^+ \fun \tilde{\mdl A}_i^+$ be the definable function given by $x \efun x^{a_i} \coloneqq \exp(a_i \log x)$ on $\tilde{\mdl A}_i^+$ and $0$ everywhere else. Thus $\zeta_i$ may be understood as an ``infinitely powered'' monomial relative to the structures below $\tilde{\mdl A}_i$.  Let $T_i$ be the theory of $(\mdl A_i, \zeta_1, \ldots, \zeta_i)$, which is power-bounded, and $T_\flat  = \bigcup_i T_i$. If $a$ is an element in a model of $T_\flat$ then we write $a_\zeta$ for the sequence $(\zeta_i(a))_i$.

\begin{thm}\label{main:trace}
Suppose that $\mdl R$ expands to a $T_\flat$-model $\mdl R_\flat$ and $\OO_\flat$ is a $T_\flat$-convex subring of $\mdl R_\flat$. Let $(\mdl R'_\flat, \OO'_\flat)$ be an elementary substructure of $(\mdl R_\flat, \OO_\flat)$. Suppose that there exists a $u \in \OO^+_\flat$ with $u > \OO'_\flat$. Then every subset of $(\mdl R'_\flat)^m$ that is $a$-definable in $(\mdl R'_\flat, \OO'_\flat)$  is the trace of a subset of $\mdl R^m$ that is $(a, u_\zeta)$-definable in $\mdl R$.
\end{thm}

Note that the trace is given without using the functions $\zeta_i$, $i > 0$. Actually, instead of $u_\zeta$, the sequence $(\zeta_i(u))_{0 \leq i < m}$ suffices for the conclusion.
This result holds for a large class of \omin-minimal theories, see Hypothesis~\ref{hyp:ord} and Remark~\ref{prog:proto}.

So far we have been using the mainstream model-theoretic language to motivate the discussion. But the origin of our inquiry lies in real geometry and topology, in particular in the effort of generalizing a construction in \cite{fichou:shiota:pui}. The overarching question this construction attempts to address may be simply understood as follows: Could there be a (singular) homology theory for the category of sets definable in $(\mdl R, \OO)$?

Let $\MM$ denote the maximal ideal of $\OO$ and $\UU$ the group $\OO \mi \MM = \OO^\times$ of units of $\OO$. Suppose that  $A$ is a definable set in $(\mdl R, \OO)$   of the form $\MM^m \cap f^{-1}(t)$, where $t \in \MM^+$ is sufficiently small and the function $f : \mdl R^m \fun \mdl R$ is definable in $\mdl R$ (one may replace $\mdl R$ with a Nash or definable manifold, but we are not concerned with generality here). A singular homology $\tilde H_*(A)$ is introduced in \cite{fichou:shiota:pui} by considering semialgebraic simplices contained in $A$. Then it is shown that, for all sufficiently small $u \in \UU^+$, $\tilde H_*(A)$ is naturally isomorphic to the homology $H_*(A_u)$ of the semialgebraic set $A_u = [-u, u]^m \cap f^{-1}(t)$; we speak of ``the homology'' of $A_u$ because all the usual ones coincide on (definably) compact semialgebraic sets. Also see the discussion in \cite[\S~5.2]{Fich:Yin:mot} for a decategorified version of this, that is, the Euler characteristic of $A$, which is defined using motivic integration, is equal to that of $A_u$. The moral to be drawn  here is that $A_u$ is indistinguishable from $A$ relative to the semialgebraic sets of interest and hence, retroactively, we may define $\tilde H_*(A)$ directly as $H_*(A_u)$.

Let $(\mdl R', \OO') \preccurlyeq (\mdl R, \OO)$ and suppose that $A$ is an $\mdl R'$-definable set in  $(\mdl R, \OO)$.  The idea now is to find a semialgebraic set $P$ in $\mdl R$ that behaves sufficiently like $A$ relative to the $\mdl R'$-definable  semialgebraic sets, and then use $P$ to construct some sort of semialgebraic homology of $A$.  Of course, in this endeavor, we need to first make precise what ``sufficiently like $A$'' means. In light of Remark~\ref{HNF}, this can be achieved by piecing together suitable intervals if $A \sub \mdl R$, see Example~\ref{exam:approx} below, which may be easily generalized to certain special subsets of $\mdl R^n$, including those of the form discussed above. However, we should not expect it to work for $\mdl R'$-definable sets in $(\mdl R, \OO)$ in higher dimensions. Here is why.

Consider the set $A = \MM \times \OO$. Naively, we may wish to find a semialgebraic sequence of rectangles $B_a \times C_a$, $a \in \mdl R^+$, such that, say, $A = \bigcup_{a \in \MM^+} B_a \times C_a$. In other words, $B_a$ and $C_a$, $a \in \mdl R^+$, are semialgebraic sequences of intervals that approach $\MM$ and $\OO$ from within, respectively, as $a$ gets larger and larger inside $\MM^+$. Moreover, and this is the crucial point,  they approach their targets at comparable rates, which just means that there is a semialgebraic function $s: \mdl R \fun \mdl R$ such that $s(\MM^+)$ is cofinal in $\OO^+$ and $A = \bigcup_{a \in \MM^+} B_a \times C_{s(a)}$. Then we declare that $B_a \times C_{s(a)}$ is sufficiently like $A$ for all sufficiently large $a$ in $\MM^+$. However, by a curious fact in the theory of \T-convexity (see Lemma~\ref{Ocon}), such a function $s$ cannot exist; here we need to require that $T$ be power-bounded. Dually,  $A = \bigcap_{a \in \MM^+} B_a \times C_{s(a)}$ is also not possible.

We try another construction. Surely it is possible that $B_a$ approaches $\MM$ from  outside and $C_a$ approaches $\OO$ from within at comparable rates, and hence $B_a \times C_{1/a}$ does approach $\MM \times \OO$ as  $a$ gets smaller and smaller in $\UU^+$. The index set $\UU^+$ is a more natural choice than $\MM^+$ for reasons that will become clear in future applications (to begin with, the boxes $B_a \times C_{1/a}$ may be thought of as indexed over the residue field $\K$ of $\OO$, which is also a model of $\RCF$). However,  in this case $A$ cannot be written simply as a union or an intersection. Thus, to get at what  ``sufficiently like $A$'' means, it is natural to ask a different question:  Is it true that, for sufficiently small $a \in \UU^+$ and $\mdl R'$-definable semialgebraic set  $A'$, if $A' \sub A$  then $A' \sub B_a \times C_{1/a}$ and, dually, if $A \sub A'$ then $B_a \times C_{1/a} \sub A'$?

Unfortunately, this is not possible again. To see it, we first assume that $B_a$, $C_{1/a}$ are just the open intervals delimited by their indices. Let $E'$ be an $\mdl R'$-definable interval containing $\OO$. Let $f : E' \fun \mdl R$ be the continuous function such that $f \rest [-1, 1]$ is the constant function $1$ and, outside $[-1, 1]$, $f(x) = 1 / x^2$. Let $E$ be the cell $(-f, f)_{E'}$. Then $A \sub E$. However, $f$ decreases faster than the function $1/\abs x$ near either end of $\OO$, and hence $B_a \times C_{1/a}$ cannot be contained in $E$ for sufficiently small $a \in \UU^+$. Now observe that we cannot get around this issue by replacing $B_a \times C_{1/a}$ with $B_a \times C_{s(a)}$, where $s: \mdl R \fun \mdl R$ is a semialgebraic function, the reason is simply that no such $s$  can decrease fast enough so to be dominated by all $\mdl R'$-definable semialgebraic functions, unless we pass to an extension of $\RCF$ and add such a function. This is the reason why we have introduced the ``fast'' power function $x^{a_1}$ above. Generally speaking,  such an extension does solve the problem for all $A \sub \mdl R^2$, but if $A \sub \mdl R^3$ then  we are back to square one and hence need to add an even ``faster'' power function $x^{a_2}$, and so on.

On second thought, in the first construction above, what if the intervals $C_a$ in the union are indexed over the more natural set $\UU^+$ instead of $\MM^+$? Let us test this modification on the complement $A^\cpt$ of $A$. Consider its partition into the three sets $\MM^\cpt \times \OO^\cpt$, $\MM^\cpt \times \OO$, and $\MM \times \OO^\cpt$. Let $B_{1/r} \times C_{1/s}$, $(r,s) \in (\mdl R^+)^2$, be a semialgebraic family of rectangles (centered at $+\infty$) that approach $\MM^\cpt \times \OO^\cpt$ from within as $r \in \UU^+$ and  $s \in \MM^+$ become larger and larger. Similarly, we construct rectangles $B_{1/r} \times B^\cpt_{r}$, $r \in \mdl R^+$, for $\MM^\cpt \times \OO$ and rectangles $C^\cpt_{s} \times C_{1/s}$, $s \in \mdl R^+$, for $\MM \times \OO^\cpt$. For each $(r, s) \in \UU^+ \times \MM^+$, let
\[
BC(r, s) = (B_{1/r} \times C_{1/s}) \cup (B_{1/r} \times B^\cpt_{r}) \cup (C^\cpt_{s} \times C_{1/s}).
\]
If we wish to declare that some $BC(r, s)$ is sufficiently like $A$ then it is reasonable to demand that it contains all the $\mdl R'$-definable  semialgebraic sets $D$ contained in $A^\cpt$. But this is untenable. To see it, suppose for contradiction that, for all sufficiently large $r \in \UU^+$ and  $s \in \MM^+$, $BC(r, s)$ contains all such $D$. For all $(r, s) \in \UU^+ \times \MM^+$, we have
\[
B_{1/r} \times C_{1/s} \sub B_{1/r} \times B_{r} \dand C^\cpt_{s} \times C_{1/s} \sub B^\cpt_{1/r} \times B_{r}.
\]
For  $r \in \UU^+$, let
\[
B(r) =  (B_{1/r} \times B_{r}) \cup (B_{1/r} \times B^\cpt_{r}) \cup (B^\cpt_{1/r} \times B_{r}),
\]
that is, $B(r) = (B^\cpt_{1/r} \times B^\cpt_{r})^\cpt$. So $BC(r, s) \sub B(r)$ for all $(r, s) \in \UU^+ \times \MM^+$. So, for all sufficiently large $r \in \UU^+$ and all  $\mdl R'$-definable  semialgebraic set $D \sub A^\cpt$,  we have $D \sub B(r)$, or equivalently, $B^\cpt_{1/r} \times B^\cpt_{r} \sub D^\cpt$. Since $A \sub D^\cpt$, this contradicts the discussion in the last paragraph.

In conclusion, for any semialgebraic set $P$ in $\mdl R$, we propose to define ``$P$ is sufficiently like $A$'' or ``$P$ is an approximant of $A$'' to mean that, for all $\mdl R'$-definable semialgebraic sets $A'$, if $A' \sub A$ then $A' \sub P$ and if $A \sub A'$ then $P \sub A'$. Let $(\mdl R_\flat, \OO_\flat)$, $(\mdl R'_\flat, \OO'_\flat)$ be as in Theorem~\ref{main:trace}. Then an ``approximation of $A$''  is a sequence $(P_a)_{a \in \mdl R_\flat^+}$ definable in $\mdl R_\flat$ such that $P_u$ is an approximant of $A$ for every sufficiently small $u \in \UU^+$. It follows from the discussion above that, in general, such an approximation cannot be semialgebraic.

With this terminology, Theorem~\ref{main:trace} may be recast as follows.

\begin{thm}
Every set definable in $(\mdl R_\flat, \OO_\flat)$ admits an approximation.
\end{thm}

\section{Preliminaries}

We say that a (complete) \omin-minimal theory is \emph{hypogenous} if it   extends the theory $\RCF$ of real closed fields, is power-bounded, is universally axiomatized, and admits  quantifier elimination; of course universal axiomatization and quantifier elimination can always be arranged through definitional extension. We shall adopt the notation and terminology of \cite{Yin:tcon:I} concerning the complete $\lan{T}{RV}{}$-theory $\TCVF$ of \T-convex valued fields (see \cite[\S~2]{Yin:tcon:I}), where the \omin-minimal \LT-theory $T$ is always assumed to be hypogenous; reminders will be provided along the way as we proceed.

The reader is also referred to the opening discussions in \cite{DriesLew95, Dries:tcon:97} for a more detailed introduction to the general theory of \T-convex valued fields and a summary of some fundamental results. Note that, in those papers, how the valuation is expressed is somewhat inconsequential. In contrast, we shall work exclusively with the two-sorted language $\lan{T}{RV}{}$, although the discussion below essentially only concerns definable sets in the valued field itself and hence other languages such as $\lan{}{convex}{}$ are equally effective; see \cite[\S~2]{Yin:tcon:I}, in particular, \cite[Example~2.8]{Yin:tcon:I} for a quick grasp of the central features of this language $\lan{T}{RV}{}$. The main reason for this choice is that such a language is a part of the basic setup for any Hrushovski-Kazhdan style integration, which will play a role in future applications of the result presented here. Informally and for all practical purposes, the language $\lan{T}{RV}{}$ may be viewed as an extension of the language $\lan{}{convex}{}$.

As usual, we work in a sufficiently saturated $\TCVF$-model
\[
\mdl R_{\rv} = (\VF, \RV, <, \rv, \ldots)
\]
with a fixed small substructure $\mdl S$. Note that $\mdl S$ is regarded as a part of the language now and hence, contrary to the usual convention in the model-theoretic literature, ``$\0$-definable'' or ``definable'' in $\mdl R_{\rv}$  means ``$\mdl S$-definable'' instead of ``parametrically definable'' if no other qualifications are given. The reason for fixing the space of parameters, mainly for the construction of motivic integrals, is explained at the beginning of \cite[\S~3]{Yin:tcon:I}.   We also require that $\mdl S$ be $\VF$-generated,  that is, the map $\rv$ is surjective in $\mdl S$,  and  $\Gamma(\mdl S)$ be nontrivial. Since  $\TCVF$ admits quantifier elimination, these two conditions guarantee that $\mdl S$ is an elementary substructure of $\mdl R_{\rv}$ and hence every definable set contains a definable point (see \cite[Hypotheses~4.1 and 5.11]{Yin:tcon:I}).

By a definable set in $\VF$ we mean a definable subset in $\VF$, by which we just mean a subset of $\VF^n$ for some $n$; similarly for other (definable) sorts or even structures in place of $\VF$ that have been clearly understood in the context.

% such as $\RV$ (the subscript here indicates the inclusion of the middle element $0 = \rv(0)$ in the $\RV$-sort, similarly for other totally ordered sets with a distinguished element \yimu{may be not necessary???}), the residue field  $\K$, the maximal ideal $\MM$, or any substructure $\mdl N$ of $\mdl R_{\rv}$. In particular,  a definable set without further qualification means a definable set in $\mdl R_{\rv}$, that is, a  definable subset of $\VF^n \times \RV^m_0$ for some $n, m \in \N$.\yimu{also not needed}

%We also choose  a section $\bfk$  of the residue map $\res : \VF \fun \K$, that is, $\bfk$ is a maximal subfield of $\OO$, such that $\bfk_{\mdl S} \coloneqq \bfk(\mdl S)$ is a  maximal subfield of $\OO(\mdl S)$ (this section $\bfk$ is not a part of the formal language).We shall also need the Hardy field $\mdl H$ of germs at $+\infty$ of functions $f : \bfk_{\mdl S}\fun \bfk_{\mdl S}$ that are definable in the \T-model $\bfk_{\mdl S}$. By the discussion in \cite[\S 5]{DMM94}, $\mdl H$ is a \T-model and is in fact generated by $\bfk_{\mdl S}$ and $x$, where $x$ stands for the germ of the identity function.  Thus $\mdl H$ is isomorphic to any \T-submodel of $\mdl S$ of the form $\bfk_{\mdl S}\la x \ra_T$, where $x \in \VF(\mdl S) \mi \OO(\mdl S)$.  We shall fix an embedding $\mdl H \fun \mdl S$ and thereby regard $\mdl H$ as a \T-submodel of $\mdl S$.

\begin{nota}[Coordinate projections]\label{indexing}
For each $n \in \N$, let $[n]$ denote the set $\{1, \ldots, n\}$. Let $A$ be a definable set in $\VF$. For $E \sub [n]$, we write $\pr_E(A)$, or even $A_E$, for the projection of $A$ into the coordinates contained in $E$. It is often more convenient to use simple standard descriptions as subscripts. For example, if $E$ is a singleton $\{i\}$ then we shall always write $E$ as $i$ and $\tilde E \coloneqq [n] \mi E$ as $\tilde i$; similarly, if $E = [i]$, $\set{k \given i \leq k \leq j }$, $\set{k \given  i < k < j}$, $\{\text{all the coordinates in the sort $S$}\}$, etc., then we may write $\pr_{\leq i}$, $\pr_{[i, j]}$, $A_{(i, j)}$, $A_{S}$, etc.

%; in particular, we shall frequently write $A_{\VF}$ and $A_{\RV}$ for the projections of $A$ into the $\VF$-sort and the $\RV$-sort coordinates.

Unless otherwise specified, by writing $a \in A$ we shall mean that $a$ is a finite tuple of elements (or ``points'') of $A$, whose length is not always indicated.
%If $a = (a_1, \ldots, a_n)$ then, for all $1 \leq i < j \leq n$, following the notational scheme above, $a_i$, $a_{\tilde i}$, $a_{\leq i}$, $a_{[i, j]}$, $a_{\VF}$, etc., are shorthand for the corresponding subtuples of $a$.

We shall write $\{t\} \times A$, $\{t\} \cup A$, $A \mi \{t\}$, etc., simply as $t \times A$, $t \cup A$, $A \mi t$, etc., when it is clearly understood that $t$ is an element and hence must be interpreted as a singleton in these expressions.

For $a \in A_{\tilde E}$, the fiber $\set{b \given ( b, a) \in A } \sub A_E$ over $a$ is often denoted by $A_a$. The distinction between the two sets $A_a$ and $A_a \times a$ is often immaterial, in which case they shall be tacitly identified. In particular, given a function $f : A \fun B$ and $b \in B$, the pullback $f^{-1}(b)$ is sometimes just written as $A_b$ as well. This is a special case since functions are identified with their graphs. This notational scheme is especially useful when $f$ has been clearly understood in the context and hence there is no need to spell it out all the time.
\end{nota}

The value group $\Gamma$ of $\mdl R_{\rv}$ is written multiplicatively and the associated valuation map $\vv : \VF^\times \fun \Gamma$ are signed (see \cite[Remark~2.8]{Yin:tcon:I} for an explanation), and the traditional valuation map is written as $\abv: \VF^\times \fun \abs{\Gamma}$; we will have no use of $\vv$ in this paper, though.  The residue map $\OO \fun \K$ is denoted by $\res$ and is extended by setting $\res(a) = 0$ for all $a \in \VF \mi \OO$.

\begin{lem}\label{Ocon}
Let $f : \OO \fun \VF$ be a definable function. Then for some $\gamma \in \abs{\Gamma} \cup \infty$ and $a \in \OO$ we have $\vv(f(b)) = \gamma$ for all $b > a$ in $\OO$.
\end{lem}
\begin{proof}
See \cite[Proposition~4.2]{Dries:tcon:97}.
\end{proof}

Note that this does not hold if $T$ is not power-bounded.

A definable function $f$ is \emph{quasi-\LT-definable} if it is a restriction of an \LT-definable function (with parameters in $\VF(\mdl S)$, of course).

\begin{lem}[{\cite[Lemma~3.3]{Yin:tcon:I}}]\label{fun:suba:fun}
Every definable function $f : \VF^n \fun \VF$ is piecewise quasi-\LT-definable; that is, there are a definable finite partition $A_i$ of $\VF^n$ and \LT-definable functions $f_i: \VF^n \fun \VF$ such that $f \rest A_i = f_i \rest A_i$ for all $i$.
\end{lem}

\begin{defn}[$\vv$-interval]
Let $\ga$, $\gb$ be two (valuative) discs, not necessarily disjoint. The subset $\ga < x < \gb$ of $\VF$, if it is not empty, is called an \emph{open $\vv$-interval} and is denoted by $(\ga, \gb)$, whereas the subset
\[
\set{a \in \VF \given x \leq a \leq y \text{ for some $x \in \ga$ and $y \in \gb$} }
\]
if it is not empty, is called a \emph{closed $\vv$-interval} and is denoted by $[\ga, \gb]$. The other $\vv$-intervals $[\ga, \gb)$, $(-\infty, \gb]$, etc., are defined in the obvious way, where $(-\infty, \gb]$ is a closed (or half-closed) $\vv$-interval that is unbounded from below.

Let $A$ be such a $\vv$-interval. The discs $\ga$, $\gb$ are called the \emph{end-discs} of $A$. If $\ga$, $\gb$ are both points in $\VF$, which may be regarded as closed discs with radius $\infty$, then of course we just say that $A$ is an interval, and if $\ga$, $\gb$ are both $\RV$-discs, that is, discs of the form $\rv^{-1}(t)$, then we say that $A$ is an $\RV$-interval, etc. If $A$ is of the form $(\ga, \gb]$ or $[\gb, \ga)$, where $\ga$ is an open disc and $\gb$ is the smallest closed disc containing $\ga$, then $A$ is called a \emph{half thin annulus}.

Two $\vv$-intervals are \emph{disconnected} if their union is not a $\vv$-interval.
\end{defn}

It is straightforward to check that every $\vv$-interval admits a unique presentation, that is, there is only one way to write it as such (the empty set is not a $\vv$-interval by definition). Thus  we may speak of  the \emph{type} of a $\vv$-interval, which is determined  by the attributes of its two ends --- open or closed or unbounded --- and the attributes of its two end-discs --- open or closed or a point.

Evidently a $\vv$-interval is definable if and only if its end-discs are definable.

A definable disc $\ga$ is a closed $\vv$-interval $[\ga, \ga]$ and may be coded by a definable element in $\VF^3$ whose first coordinate is a definable point in $\ga$, whose second coordinate gives $\rad(\ga) \in \abs{\Gamma} \cup \infty$, the radius of $\ga$, and whose last coordinate is either $1$ or $-1$, according to whether $\ga$ is open or closed. Similarly, a definable $\vv$-interval may be coded by a definable element in $\VF^8$,  which in addition records whether the lower and the higher ends of the $\vv$-interval are open or closed or unbounded.

\begin{rem}\label{HNF}
Since the \omin-minimal theory $T$ is power-bounded, we have an important tool called \emph{Holly normal form} (henceforth abbreviated as HNF), that is, every definable subset of $\VF$ is a unique union of finitely many pairwise disconnected definable $\vv$-intervals. This is obviously a generalization of the \omin-minimal condition. It is equivalent to the so-called valuation property (see \cite[\S~7]{Dries:tcon:97}), for which the polynomially bounded case is first established in \cite[Proposition~9.2]{DriesSpei:2000}  and the general case (power-bounded) in \cite{tyne}.
\end{rem}

\begin{defn}[Orientation of a $\vv$-interval]\label{fun:iota}
Let $A$ be a $\vv$-interval. We assign a number $\iota'(A)$ from the set $\{-1, 0, 1\}$ to the lower end of $A$ as follows. If it is an open end with an open end-disc or a closed end with a closed end-disc then $\iota'(A) = 1$. If the end-disc is a point in $\VF \cup \{\pm \infty \}$ then $\iota'(A) = 0$. In the remaining two cases $\iota'(A) = - 1$. Similarly, we assign such a number $\iota''(A)$ to the higher end of $A$. Let $\iota(A) = \iota'(A)\iota''(A)$. We say that $A$ is \emph{oriented} if $\iota(A) \neq -1$.
\end{defn}

For instance, a  $\vv$-interval $A$ is oriented if it is an open disc or a closed disc or a half thin annulus. More generally, let $\ga$ be an open disc and $\gb$ a closed disc with $\rad(\ga) \geq \rad(\gb)$. We have that if $A$ is of the form $[\ga, \gb]$ then $A = \ga \cup (\ga, \gb]$, and if $A$ is of the form $(\ga, \gb)$ then $A = (\ga, \gd] \cup (\gd, \gb)$, where $\gd$ is the smallest closed disc containing $\ga$, and so on. In fact, every $\vv$-interval is a union of (at most two) oriented $\vv$-intervals. This is a bit tedious to check and is left to the reader.

The following lemma is an analogue of \omin-minimal monotonicity.

\begin{lem}[Monotonicity, {\cite[Corollary~3.4]{Yin:tcon:I}}]\label{mono}
Let $A \sub \VF$ and $f : A \fun \VF$ be a definable function. Then there is a definable finite partition of $A$ into oriented $\vv$-intervals $A_i$ such that every $f \rest A_i$ is quasi-\LT-definable, continuous, and monotone. Consequently, each set $f(A_i)$ is a $\vv$-interval too.
\end{lem}

\begin{ter}[Cell]\label{cell:stip}
According to \cite[Definition~4.5]{mac:mar:ste:weako}, the notion of a cell in a weakly \omin-minimal structure is almost the same as that in an \omin-minimal structure except that the bounding functions in the former may take imaginary elements as values (consequently the continuity condition for the bounding functions no longer makes sense and is dropped). The theory $\TCVF$ is  weakly \omin-minimal; for this we actually do not need to assume that $T$ is power-bounded, see \cite[Corollary~3.14]{DriesLew95}.

In our setting, due to the presence of HNF, we may instead require that the bounding functions $\bar f_i$, $\bar g_i$ in each step  of the construction of a cell $A \sub \VF^n$  form a $\vv$-interval $I_a$, all  of the same type, over each point $a \in \pr_{\leq i-1}(A)$ in the obvious sense. Furthermore,  the bounding functions $\bar f_i$, $\bar g_i$ take the form of restrictions of (not necessarily unique) \LT-definable functions $f_i, g_i : \VF^{i-1}  \fun\VF^4$, which determine the type of the $\vv$-intervals in question.  This will be what we mean by a \emph{quasi-cell} below, and accordingly $f_i$, $g_i$ its \emph{upper} and \emph{lower characteristic functions}. A \emph{cell} is a quasi-cell in which all the $\vv$-intervals $I_a$ are oriented.
\end{ter}
%\begin{itemize}
%\item an open oriented $\vv$-interval (possibly an open interval) or
%\item a disc (possibly a point in $\VF$) or
%\item a half thin annulus,
%\end{itemize}and there is an \LT-definable (\omin-minimal) cell $A'$ containing $A$ such that  $g_i \rest A'$ is continuous.

Quasi-cell and hence cell decompositions  hold accordingly. This follows from the discussion  above, in particular, HNF and Lemma~\ref{fun:suba:fun}.

Often we will omit the last two coordinates in the target of a characteristic function $f$ of a quasi-cell $A$ and write it more conveniently as a pair of \LT-definable functions $f', f'' : \VF^{i} \fun \VF$ with $f' \leq f''$ such that, for all $a \in \pr_{\leq i}(A)$, $f'(a)$ is contained in the end-disc in question and $\abv(f''(a) - f'(a))$ its radius (if the end in question is unbounded then $f'$, $f''$ are taken to be the empty function). We will then  also refer to $f'$, $f''$ as characteristic functions of $A$.

\begin{lem}\label{vbox}
Let $A \sub \VF^n$ be a definable set. Then there are a quasi-cell decomposition $(A_i)_i$ of $A$ and  \LT-definable bijections $\sigma_i^j : \VF^{j} \fun \VF^j$ with $\sigma_i^j \circ \pr_{\leq j} = \pr_{\leq j} \circ \sigma_i^n$, $1 \leq j \leq n$, such that each $\sigma_i^n(A_i)$ is a $\vv$-box, that is, a set of the form $\prod_j I_j$ with each $I_j$  a $\vv$-interval.
\end{lem}
\begin{proof}
We do induction on $n$. The base case $n=1$ is rather trivial.   For the inductive step, upon further decomposition, we may assume that  $A' = \pr_{<n}(A)$ is already a $\vv$-box (the bijection in question is just the identity function on $\VF^{n-1}$) and $A$ is a quasi-cell such that, for each $a \in A'$, one of the end-discs of $A_a$ is a point. Due to the presence of characteristic functions, without loss of generality, we may assume that this endpoint of $A_a$ is in fact $+\infty$. Now there are several cases to consider. Since they are all quite similar, it is enough to deal with the case that each $A_a$ is of the form $(\ga_a, +\infty)$, where $\ga_a$ is an open disc. Then the claim is clear since, using a (lower) characteristic function of $A$ again, we can construct an $a$-\LT-definable bijection between $(\ga_a, +\infty)$ and $(\MM, +\infty)$, uniformly for all $a \in A'$.
\end{proof}

\begin{defn}[$\res$-contractions]\label{defn:corr:cont}
A function $f : A \fun B$ between two sets in $\VF$ is \emph{$\res$-contractible} if there is a (necessarily unique) function $f_{\downarrow} : \res(A) \fun \res(B)$, called the \emph{$\res$-contraction} of $f$, such that $(\res \rest B) \circ f = f_{\downarrow} \circ (\res \rest A)$.
\end{defn}

\section{Approximating a definable set by $o$-minimal sets}
To even state the definition of an approximation, the hypogenous theory $T$ needs to admit an ascending sequence of hypogenous extensions satisfying certain constraints. The precise condition is formulated as follows.

\begin{hyp}\label{hyp:ord}
There is an ascending sequence of hypogenous $\lan{T_i}{}{}$-theories $T_i$, $i \in \N$, with $T_0 = T$ such that, for all $i$,   $x > \KK_i$ for some $x \in \KK_{i+1}$, where $\KK_i$ is the field of exponents of $T_i$.  We call the sequence $(T_i)_i$   a \emph{power progression at $T$}.
\end{hyp}

Here we may write $\KK_i \sub \KK_{i+1}$ for all $i$. Let $\KK = \bigcup_i \KK_i$, $\lan{T_\flat}{}{} = \bigcup_i \lan{T_i}{}{} $, and $T_\flat = \bigcup_i T_i$. It follows that $T_\flat$ is a hypogenous $\lan{T_\flat}{}{}$-theory whose field  of exponents is $\KK =\bigcup_i \KK_i$. Also, if the prime model of $T_i$ is denoted by $\mdl P_i$ then the prime model $\mdl P_\flat$ of $T_\flat$ equals $\bigcup_i \mdl P_i$. Since $\KK$ is cofinal in $\mdl P$, it  follows that the convex hull of $\KK$ in any $T_\flat$-model is $T_\flat$-convex.

%\begin{exam}\label{exam:rcf}
%Write the real field $\R$ as $\bar \R$ to emphasize that it is a model of $\RCF$. By Wilkie's Theorem, the theory of the structure $(\bar \R, \exp)$ is \omin-minimal. Let $\tilde{\mdl A}_i = (\mdl A_i, \exp)$ be an ascending sequence of proper elementary extensions of $(\bar \R, \exp)$:
%\[
%(\bar \R, \exp) = \tilde{\mdl A}_0 \prec \tilde{\mdl A}_1 \prec \ldots \prec \tilde{\mdl A}_n \prec \ldots
%\]
%For each $i$, let $a_i \in \tilde{\mdl A}_i^+$ with $a_i > \tilde{\mdl A}_{i-1}$ and $x^{a_i} : \tilde{\mdl A}_i^+ \fun \tilde{\mdl A}_i^+$ be the definable function given by $x \efun \exp(a_i \log x)$. Thus $x^{a_i}$ may be understood as an ``infinitely powered'' monomial relative to the structures below $\tilde{\mdl A}_i$.  Let $T_i$ be the theory of $(\mdl A_i, x^{a_1}, \ldots, x^{a_i})$. It is well-known that $T_i$ is power-bounded.
%\end{exam}For convenience, we may assume that $\lan{T_0}{}{}$ is a functional language, which just means that all its primitives, except the binary relation $<$, are function symbols.

There is an ample supply of power progressions. The general procedure for manufacturing such sequences below is based on the results in \cite[\S~6]{fostThesis}, which generalizes the work in \cite{DriesSpei:2000}.

\begin{rem}\label{prog:proto}
Suppose that $T_0$ is a hypogenous $\lan{T_0}{}{}$-theory and $\mdl P_0 = (P_0, <, \ldots)$ is its prime model. Let $\KK_0$ be the field of exponents of $T_0$. Assume that $\KK_0$ is cofinal in $\mdl P_0$ and $\mdl P_0$ defines a  restricted exponential function $E : [-1, 1] \fun P_0$   with $\ddx E(0) = 1$.
For future applications, it is perhaps simpler to just assume that $P_0 = \R$ and $T_0$ contains the theory $T_{\text{an}}$ of the real field with all restricted analytic functions as defined in \cite{DMM94}. Then $\mdl P_0$ can be expanded to an \omin-minimal $\lan{T_0}{}{}(\exp, \log)$-structure $\tilde{\mdl P}_0 = (\mdl P_0, \exp, \log)$, where $\exp$ and $\log$ are new function symbols, such that  $\exp$ is interpreted in  $\tilde{\mdl P}_0$  as the exponential function extending $E$ (characterized by the property $\ddx\exp = \exp$) and $\log$ its inverse. Furthermore, $\Th(\tilde{\mdl P}_0)$ admits quantifier elimination and a universal axiomatization, but we shall not need this.

Let $\tilde{\mdl P}_i = (\mdl P_i, \exp, \log)$, $i \in \N$, be an ascending sequence of elementary extensions of $\tilde{\mdl P}_0$ such that, for each $i$, there is an $a_i \in \tilde{\mdl P}_i^+$ with $a_i > \tilde{\mdl P}_{i-1}$. Let $\tilde{\mdl P}  = (\mdl P, \exp, \log) = \bigcup_i \tilde{\mdl P}_i$ and $x^{a_i} : \tilde{\mdl P}^+ \fun \tilde{\mdl P}^+$ be the function given by $x \efun \exp(a_i \log x)$. We claim that the structure $\mdl A_i = (\mdl P, x^{a_j})_{j \leq i}$ is power-bounded. Suppose for contradiction that this is not the case. Then, by the same argument as in \cite[Example~1.4]{MR1428013}, there would be $k_j \in \KK_0$, $1 \leq j \leq i$, with $0 \ll k_1 \ll \ldots \ll k_i$ (here $a \ll b$ means that  $b$ is sufficiently larger than $a$) such that  the structure $(\mdl P_0, x^{k_1}, \ldots, x^{k_i})$ is exponential, which directly contradicts the assumption that $T_0$ is power-bounded.
After extending by definitions, we obtain a power progression $(\Th(\mdl A_i))_i$ at $T_0$.
\end{rem}
%Observe that $\bigcup_i \Th(\mdl A_i)$ is a hypogenous extension of $T_0$ whose field of exponents is cofinal in its prime model. Therefore, upon iteration, it is also possible to construct a power progression $(T_i)_i$ at $T_0$ such that  every $T_i$ has this property. This suggests that if we retroactively  include the two extra conditions on $\mdl P_0$  above in the definition of a  hypogenous theory then there is a somewhat natural, though not unique,  power progression at $T_0$, namely one in which each $T_i$ is a
%any power progression $(T_i)_i$ may be extended to a sequence of exponential \omin-minimal theories of the form $(\Th(\mdl P_i, \exp, \log))$

From here on we assume that $\mdl R_{\rv}$ is a sufficiently saturated $\ticvf \flat$-model and $\mdl S$ is a (small) elementary substructure. So, by Hypothesis~\ref{hyp:ord}, every definable set is $\lan{T_i}{RV}{}$-definable for some $i$.

%For simplicity, an $\lan{T_\flat}{}{}$-definable set is sometimes referred to as an \omin-minimal set.
%The $\lan{T_i}{}{}$- and  $\lan{T_i}{RV}{}$-reducts  of $\mdl R_{\rv}$ will all be denoted by $\mdl R_{\rv}$ when no confusion can arise; similarly for the reducts of  $\mdl S$, etc.

Since a  point in $\VF$ is definable if and only if it is in $\VF(\mdl S)$, such a definable point  is indeed $\lan{T_i}{}{}$-definable for every $i$ (since parameters in $\mdl S$ are allowed). It follows from this and HNF that all the $\lan{T_i}{RV}{}$-reducts of $\mdl R_{\rv}$ have the same definable subsets of $\VF$. This of course is not the case for definable subsets of $\VF^n$ if $n > 1$.

\begin{defn}[Growth representative]\label{alcon}
An $\lan{T_{i}}{}{}$-definable function $f :  (b, +\infty) \fun \VF$  (or rather the corresponding element in the Hardy field associated with the $\lan{T}{}{}$-reduct of $\mdl R_{\rv}$) is  \emph{$T_i$-infinite} and its multiplicative inverse $1 / f$ at $+\infty$ (taken in the Hardy field, to be more precise) \emph{$T_i$-infinitesimal} if it dominates all (parametrically) $\lan{T_{i-1}}{}{}$-definable functions at $+\infty$. For instance,  any power function $x^{a_{i}}$, $a_{i} \in \KK_{i} \mi \KK_{i-1}$, is $T_i$-infinite. We shall also need the fact that if $f$ is $T_i$-infinite then, at $+\infty$, its compositional inverse $f^{-1}$, such as the power function $x^{1/a_{i}}$ if $f$ is $x^{a_{i}}$, dominates all  constant functions and is dominated by every (parametrically) $\lan{T_{i-1}}{}{}$-definable increasing function; such a function  is called \emph{infinite-$T_i$-constant} and its multiplicative inverse \emph{infinitesimal-$T_i$-constant}. More generally, a function $f : (b, +\infty) \fun \VF$ is \emph{$T_i$-infinite at $b$}, etc., if $f \circ \frac 1 {x-b}$ is $T_i$-infinite, etc.

For each $i \geq 0$, let $\zeta_{i} :  \VF \fun \VF$  be a $T_i$-infinite  function given by a function symbol of $\lan{T_{i}}{}{}$ and suppose that $\zeta_{i} \rest \VF^+$ is strictly increasing.  Then the sequence $(\zeta_i)_i$ is called a \emph{growth representative} for  $(T_i)_i$.
\end{defn}

Fix a growth representative $(\zeta_i)_i$  for $(T_i)_i$. Observe that the sequence $(\zeta_i, \zeta_{i+1}, \ldots)$ is a growth representative for the power progression $(T_i, T_{i+1}, \ldots)$. Also, the  $\res$-contraction $\zeta_{i \downarrow} :  \K \fun \K$ of  $\zeta_i$ is just the function  defined by the same function symbol in the $T_i$-model $\K$, at least at $+ \infty$, and hence, in $\K$, $\zeta_{i \downarrow}$ is $T_{i}$-infinite, $1/ \zeta_{i \downarrow}$ is $T_{i}$-infinitesimal, $\zeta_{i \downarrow}^{-1}$ is infinite-$T_{i}$-constant, etc.

%For simplicity, many definitions and results below are only formulated with respect $T$, but they remain the same if $T$ is replaced by any $T_i$ and hence \T-convexity by $T_i$-convexity, since they are just specializations of the general theory for hypogenous \omin-minimal theories as developed in \cite{Yin:tcon:I}. Eventually we will need to work in situations where more than one of these theories  $T_i$ are involved. Then we will state clearly which sets are defined with respect to which $T_i$ when it becomes necessary.

%\begin{nota}
%The following shorthand will greatly reduce verbosity in the discussion below. For $B \sub \VF^+$, we write $b \ssin B$ to mean that $b$ is contained in $B$ and is sufficiently small for certain condition to hold. Exactly what the condition is will be supplied by the context, hence wherein the meaning of ``sufficiently small'' should be clear.
%\end{nota}

For any  $b \in \VF$,   the sequence $(\zeta_0(b), \zeta_1(b), \ldots)$ is denoted by $b_\zeta$. If $\MM < u < \VF^+(\mdl S) \mi \MM$ then we write $u \ssin \UU^+$; such an element exists since  $\mdl R_{\rv}$ is sufficiently saturated.

Let $A \sub \VF^n$ be an $\lan{T}{RV}{}$-definable set. Here is the key definition of the paper.

\begin{defn}[Approximation]\label{def:tapp}
Let $P \sub \VF^{n} \times \VF^+$ be an $\lan{T_i}{}{}$-definable set with $\pr_{>n}(P) = \VF^+$  such that each fiber $P_b$, $b\in \VF^+$, is $b_\zeta$-$\lan{T}{}{}$-definable. We say that $P$ is a \emph{$T_i$-approximation} of $A$ if, for some $u \ssin \UU^+$ (hence for all $u \ssin \UU^+$, see Remark~\ref{some:all} below), the fiber $P_u$ is contained in all $\lan{T}{}{}$-definable sets containing $A$ and contains all $\lan{T}{}{}$-definable sets contained in $A$. Such a $u_\zeta$-$\lan{T}{}{}$-definable set $P_u$ (not necessarily a fiber of some $T_i$-approximation) shall be referred to as a \emph{\T-approximant} of $A$.
\end{defn}

We shall refer to $P$ simply as an approximation of $A$ when there is no need to emphasize that it is $\lan{T_i}{}{}$-definable.

%Similarly, by an approximant we may mean a $T_i$-approximant   if the growth representative in question is, for instance, the sequence $(\zeta_i, \zeta_{i+1}, \ldots)$.

The following  more convenient functional notation for approximations will be used throughout:  $P$ may be  thought of as a function on $\VF^+$ with  $P(b)  = P_{b}$.

\begin{exam}\label{exam:approx}
Let $\ga$ be a definable open disc with $\rad(\ga) = \gamma$, which can also be written as a closed $\vv$-interval $[\ga,\ga]$. We have pointed out above that such a set is actually $\lan{T_i}{RV}{}$-definable for all $i$.  Since $\ga$ contains a definable point, it is definably bijective to $\MM_\gamma$, the open disc of radius $\gamma$ centered at $0$. Let $a \in \VF^+$ be a definable point with $\vv(a) = \gamma$. Then
\[
 \MM_\gamma = \bigcap_{u \in \UU^+} [-a\zeta_0(u), a\zeta_0(u)] = \bigcap_{u \in \UU^+} (-a\zeta_0(u), a\zeta_0(u))
\]
and hence $\ga$ can also be written as an intersection of $\UU$-definable closed or open intervals ordered by inclusion and indexed by $\UU^+$. For  $u \ssin \UU^+$, $[-a\zeta_0(u), a\zeta_0(u)]$ or $(-a\zeta_0(u), a\zeta_0(u))$  is a \T-approximant of $\MM_\gamma$: it is contained in all definable intervals containing $\MM_\gamma$ and contains all definable intervals contained in $\MM_\gamma$ (the second clause is of course trivial, but it will not be in the general case).

Dually, if $\ga$ is a definable closed disc then it is definably bijective to $\OO_\gamma$, the closed disc of radius $\gamma$ centered at $0$, and the latter can be written as a union of the same $\UU$-definable  intervals inversely ordered by inclusion and indexed by $\UU^+$. The sets
\[
 \bigcup_{b \in \VF^+} [-a / \zeta_0(b),  a  / \zeta_0(b)] \times b \dand \bigcup_{b \in \VF^+} (- a / \zeta_0(b),  a / \zeta_0(b)) \times b
\]
are then two approximations of $\OO_\gamma$.
\end{exam}

\begin{rem}\label{some:all}
Let $u', u'' \ssin \UU^+$ and $P(u')$ be a \T-approximant of  $A$. Then, by the argument in the proof of \cite[Theorem~2.16]{Yin:tcon:I}, there is an $\lan{T_\flat}{RV}{}$-automorphism  $\sigma$ of  $\mdl R_{\rv}$ (over $\mdl S$) with $\sigma(u') = u''$ and hence $P(u'')  = \sigma(P(u'))$ is also a  \T-approximant of  $A$. It follows that, for some $i$, there is an $\lan{T_i}{}{}$-formula $\phi(x, y)$ such that, for  $\usu$, $\phi(x, u)$ defines a \T-approximant of  $A$ and hence the set defined by $\phi(x, y)$ is indeed a $T_i$-approximation of $A$.
\end{rem}

\begin{rem}\label{pieapprox}
Let $f : A \fun \VF$ be an $\lan{T}{RV}{}$-definable function. Suppose that $f_1, f_2 : \VF^n \fun \VF$ are two $\lan{T}{}{}$-definable functions that both restrict to $f$ on $A$. Then $(f_1 - f_2)^{-1}(0)$ is an $\lan{T}{}{}$-definable set containing $A$ and hence if  $P(u)$ is a \T-approximant of  $A$ then  $f_1 \rest P(u) = f_2 \rest P(u)$.
\end{rem}

\begin{lem}\label{omin:par}
Let  $(C_i)_i$ be an $\lan{T}{}{}$-definable finite partition of $\VF^n$. If $P_i$ is a $T_i$-approximation of $A \cap C_i$ then $P = \bigcup_i P_i$ is a $T_i$-approximation of $A$.
\end{lem}
\begin{proof}
Let $B$ be an $\lan{T}{}{}$-definable set contained in $A$. Then, for $u \ssin \UU^+$, $P_{i}(u)$ contains $B \cap C_i$ for every $i$ and hence $P(u)$ contains $B$. The other case is similar.
\end{proof}

\begin{lem}\label{in:out:same}
Let $P$ be an approximation of  $A$. Then, for all $\usu$ and all $\lan{T}{}{}$-definable set $B$,  $B \sub A$ if and only if $B \sub P(u)$ and, dually, $A \sub B$ if and only if $P(u) \sub B$.
\end{lem}
\begin{proof}
One direction of the first claim is by definition. For the other direction, suppose for contradiction that $B \sub P(u)$  but $B \nsubseteq A$. Choose a definable point $b \in B \mi A$. Then $\VF^n \mi b$ is an $\lan{T}{}{}$-definable set containing $A$ and hence it contains $P(u)$, which is impossible.  The second claim is similar.
\end{proof}

\begin{cor}
Let $P(u)$ be a \T-approximant of  $A$. Then $A(\mdl S) = P(u)(\mdl S)$.
\end{cor}
\begin{proof}
By Lemma~\ref{in:out:same}, a definable point belongs to $A(\mdl S)$ if and only if it belongs to $P(u)(\mdl S)$.
\end{proof}

So the ``duality'' requirement in Definition~\ref{def:tapp} gives rise to an alternative definition of an approximation as follows.

\begin{defn}
Let $P \sub \VF^{n} \times \VF^+$ be an $\lan{T_i}{}{}$-definable set with $\pr_{>n}(P) = \VF^+$  such that each fiber $P(b)$, $b\in \VF^+$, is $b_\zeta$-$\lan{T}{}{}$-definable. We say that $P$ is a \emph{$T_i$-approximation} of $A$ if, for all $u \ssin \UU^+$,  $P(u)$ contains all $\lan{T}{}{}$-definable sets contained in $A$ and $P(u)^\cpt$ contains all $\lan{T}{}{}$-definable sets contained in $A^\cpt $.
\end{defn}

Our goal is to show that $A$ admits a $T_{n-1}$-approximation. To illustrate the key idea underlying the construction, we shall first consider a special case, namely quasi-cells, even though only a subcase of it, namely $\vv$-boxes, is needed for the general case.

%Also, that every $P(b)$ is $b_\zeta$-$\lan{T}{}{}$-definable will be quite clear from the construction and the inductive hypothesis, so we will not bring this issue up anymore.

\begin{lem}\label{LT:lim}
Suppose that $A$ is a quasi-cell. Then it admits a  $T_{n-1}$-approximation.
\end{lem}
\begin{proof}
Fix any $\usu$. We proceed by induction on $n$. For the base case $n=1$, that is, $A$ is just a $\vv$-interval, the essential cases have already been checked in Example~\ref{exam:approx} and the other cases are similar. Recall Definition~\ref{fun:iota} and note that if $A$ is not oriented and $P$ is a \T-approximation of $A$ as constructed in Example~\ref{exam:approx} then $P(u)$ neither contains nor is contained in $A$.

For the inductive step, we will first deal with three distinguished cases and then explain  how to reduce all other cases to these three cases. That $P(u)$ is $u_\zeta$-$\lan{T}{}{}$-definable will be quite clear from the construction and the inductive hypothesis, so we will not bring this issue up anymore.

Let  $A' = \pr_{<n}(A)$, which is a quasi-cell. Let $P'$ be a $T_{n-2}$-approximation of $A'$.

\textbf{Case (1)}: The sets $A_a$, $a \in A'$, are all open intervals.

Recall Terminology~\ref{cell:stip}. Let $\theta_1 < \theta_2 : \VF^{n-1} \fun \VF$ be \LT-definable characteristic functions of $A$ such that, for all $a \in A'$, $A_a = (\theta_1(a), \theta_2(a))$. For each $b \in \VF^+$, let $P(b) = \bigcup_{a \in P'(b)} a \times I_{a}$, where $I_{a}$ is the open interval $(\theta_1(a), \theta_2(a))$. By Remark~\ref{pieapprox},  $P(u)$ does not depend on the choice of $\theta_1$, $\theta_2$.

We claim that $P(u)$ is contained in  every \LT-definable set $C$ containing $A$. To see this, let $C' = \pr_{<n}(C)$. Then $P'(u) \sub C'$. Since shrinking $C$ is conducive to our purpose, we may assume that there are $\lan{T}{}{}$-definable functions $g_1, g_2 : \VF^{n-1} \fun \VF$ such that $g_1 \leq \theta_1 < \theta_2 \leq g_2$ and  $C$ is of the form $(g_1, g_2)_{C'}$. This readily implies $P(u) \sub C$. Dually, if $C \sub A$ then $C' \sub P'(u)$,  and since $C \sub (\theta_1, \theta_2)_{C'}$, it is rather clear that  $C \sub P(u)$. So the function $P$ on $\VF^+$ is  a $T_{n-1}$-, in fact,  $T_{n-2}$-approximation of $A$.

\textbf{Case (2)}: The sets $A_a$, $a \in A'$, are all open discs. This and the next cases are more illuminating.

Let $\theta < \rho : \VF^{n-1} \fun \VF$ be \LT-definable characteristic functions such that, for all $a \in A'$,
\[
\theta(a) \in A_a \dand \abv(\rho(a) - \theta(a)) = \rad(A_a).
\]
Set $\delta = \rho - \theta$.    For  $b \in \VF^+$, denote the open interval
\begin{equation}\label{infinter}
 (\theta(a) - \delta(a) / \zeta_{n-1}(1/b), \theta(a) + \delta(a) / \zeta_{n-1}(1/b) )
\end{equation}
by $I_{ba} $ and set $P(b) = \bigcup_{a \in P'(b)} a \times I_{ba}$.

Again, we show that $P(u)$ is contained in  every \LT-definable set $C$ containing $A$. We may assume that $C$ is of the form $(f_1, f_2)_{C'}$, where $C'$ is as before and $f_1, f_2 : \VF^{n-1} \fun \VF$ are $\lan{T}{}{}$-definable functions with
\[
f_1 - \theta = \theta - f_2 < 0 \dand f_2 \leq \rho.
\]
Let $g :  \VF^{n-1} \fun \VF^+$ be the $\lan{T}{}{}$-definable function given by
\[
a \efun (f_2(a) - \theta(a)) / \delta(a).
\]

For each $b \in \VF^+$, let  $D_{b} = \set{g(a) \given a \in P'(b)}$ and $h : \VF^+ \fun \VF^+$ be the $\lan{T_{n-2}}{}{}$-definable function given by $b \efun \inf D_b$. Since $g(a) \in \UU^+$ for all $a \in A'$, we must have $h(u) \in \OO^+$. By monotonicity and HNF, there exists a $d \in \UU^+(\mdl S)$ such that if we write $I_d$ for the $\vv$-interval $(\MM,d)$ then $h \rest I_d$ is monotone and either $h(I_d) \sub \MM^+$ or  $h(I_d) \sub \UU^+$. Suppose for contradiction $h(I_d) \sub \MM^+$.   If $h \rest I_d$ is not strictly decreasing then clearly there is a $c \in \MM^+(\mdl S)$ with $c > h(I_d)$. If  it is strictly decreasing then, by Lemma~\ref{Ocon}, $\vv(h(I_d))$ has a maximal element and hence the same conclusion holds. On the other hand, observe that if $a \in A'$ then $g(a) > \MM$. It follows that, for any such $c \in \MM^+(\mdl S)$, the $\lan{T}{}{}$-definable set $g^{-1}(\{x > c\})$ contains $A'$ and hence it must contain $P'(u)$, which is impossible.

Now, $D_u \geq h(u) > \MM$ and hence
\[
P_{h}(u) \coloneqq \bigcup_{a \in P'(u)} a \times I_{hua} \sub C,
\]
where $I_{hua}$ is the interval defined as in (\ref{infinter}) but with $1/\zeta_{n-1}(1/b)$ replaced by $h(u)$. If  $\gu$ is an $\RV$-disc and $\MM < \gu < \VF^+(\mdl S) \mi \MM$ then $\gu$ is indeed $\code \gu$-atomic (see \cite[Notation~2.34, Definition~3.10]{Yin:tcon:I}) and hence, by \cite[Lemma~3.13]{Yin:tcon:I}, $h(\gu)$ is either a point or an open disc. Therefore, by \cite[Theorem~A]{Dries:tcon:97}, the restriction $h \rest \UU^+$  $\res$-contracts to an $\lan{T_{n-2}}{}{}$-definable function $h_{\downarrow} : \K^+ \fun \K^+$ near $0 \in \K$. Alternatively, this may be deduced from \cite[Proposition~2.20]{DriesLew95} since $h \rest I_d$ is the restriction of  a continuous $\lan{T_{n-2}}{}{}$-definable function $\VF \fun \VF$. Anyway,   by the choice of $\zeta_{n-1}$, the function $h_{\downarrow} \circ 1/x$ dominates $1 / \zeta_{(n-1)}$ at $+\infty$ in $\K$ and hence $P(u) \sub P_{h}(u) \sub C$.

We also need to show that if $C \sub A$ then it is contained in $P(u)$.  As in Case (1), this is clear since if $a \in A' \cap P'(u)$ then $A_a \sub I_{ua}$.

\textbf{Case (3)}: The sets $A_a$, $a \in A'$, are all closed discs.

This case is quite similar to Case (2), so we will be brief where details may be readily  inferred from the discussion above. We still have the two functions $\theta$, $\rho$ and the construction of the function $P$ is the same except that, in (\ref{infinter}),  $1 / \zeta_{n-1}(1/b)$ is replaced by $\zeta_{n-1}^{-1}(1/b)$.

Let $A \sub C$, where $C$ is of the form $(f_1, f_2)_{C'}$, but of course without the condition $f_2 \leq \rho$. Let $g$, $h$, etc., be as  above. So $g(a) > \OO$ for all $a \in A'$. If $h(u) > \OO$ then clearly $P(u) \sub C$. So assume $h(u) \in \OO^+$.  The function $h_{\downarrow}$ is unbounded near $0 \in \K$, for otherwise there would be a $c \in \UU^+(\mdl S)$ such that $g^{-1}(\{x > c\})$ contains $A'$ but not $P'(u)$. By the choice of $\zeta_{n-1}$, the function $h_{\downarrow} \circ 1/x$ dominates $\zeta^{-1}_{n-1}$ at $+\infty$ in $\K$ and hence $P(u) \sub P_{h}(u) \sub C$.

On the other hand, if $C \sub A$ then  $g(C') \sub \OO^+$. Since $g(C')$ is an \LT-definable set, there is a $c \in \UU^+(\mdl S)$ such that $g(C') \sub (-c, c)$. By the choice of $\zeta_{n-1}$ again, we have $\zeta_{n-1}^{-1}(1/u) > c$ and hence $C \sub P(u)$.

\textbf{The general case}: The three cases above, together with certain obvious variations, are the building blocks for the general case, that is, $A_a$, $a \in A'$, are $\vv$-intervals,  all of the same type. The key to the construction is the choice between the two functions $1/x \circ \zeta_{n-1} \circ 1/x$ and $\zeta^{-1}_{n-1} \circ 1/x$ in (\ref{infinter}).  This surely depends on the type of the $\vv$-interval $A_a$, but there are just three possibilities, for either ends of the $\vv$-intervals $A_a$, corresponding to the three cases above. Let us just  consider the lower ends.
\begin{itemize}[leftmargin=*]
  \item If $\iota'(A_a) = 1$ then $\zeta^{-1}_{n-1} \circ 1/x$ is used, as in Case(3).
  \item If $\iota'(A_a) = 0$ then  the constant function $0$ is actually used, since $\theta$ will just be the bounding function (possibly taking the values $\pm \infty$ if the end is unbounded), as in Case (1).
  \item If $\iota'(A_a) = -1$ then $1/x \circ \zeta_{n-1} \circ 1/x$ is used, as in Case (2).
\end{itemize}
Note that, though, for $\iota'(A_a) = \pm 1$,  $\zeta^{-1}_{n-1} \circ 1/x$ and $1/x \circ \zeta_{n-1} \circ 1/x$ may need to be replaced by $1/x \circ \zeta^{-1}_{n-1} \circ 1/x$ and $\zeta_{n-1} \circ 1/x$, respectively, depending on whether the end is open or closed.
\end{proof}

\section{Proof of the main theorem}

To generalize  Lemma~\ref{LT:lim} for all $\lan{T}{RV}{}$-definable sets $A \sub \VF^n$, we need to construct an $\lan{T}{RV}{}$-definable cell decomposition of $A$ of a special form. It is much easier to grasp how such a cell decomposition looks like in lower dimensions, so we first describe  a prototype with $A = \VF$.

Let $(a, b)$ be a definable open interval and $\ga$ a definable disc with $\rad(\ga) > 0$ such that either $a \in \ga$ and $b \notin \ga$ or the other way around. Thus $\ga$ induces a so-called \emph{bipartite} cell decomposition of $(a, b)$. Such cell decompositions can be classified into four types, according to whether $\ga$ is open or closed and which endpoint it contains. Let $(A_i)_i$ be a cell decomposition of $\VF$. We claim that there is an $\lan{T}{}{}$-definable cell decomposition $(I_j)_j$ of $\VF$ such that  the cell decomposition of each $I_j$ induced by $(A_i)_i$ is either trivial  or bipartite; we call such an $\lan{T}{}{}$-definable cell decomposition $(I_j)_j$ an \emph{\omin-frame} of $(A_i)_i$. To see this, let $(\ga_k)_k$ enumerate the distinct end-discs of the cells $A_i$ (so if $A_i < A_{i+1}$ are not properly disconnected then the end-discs of their opposing ends are not listed twice). Let $(I_j)_j$ be any $\lan{T}{}{}$-definable cell decomposition of $\VF$ such that each $\ga_k$ contains an endpoint of some cell $I_j$. This condition ensures that the set of the endpoints of the cells $I_j$ include all the points among the discs $\ga_k$. Let $n_j$ be the number of the cells $A_i$ such that $A_i \cap I_j \neq \0$. Clearly we are done if $\max_j n_j = 2$ (if $n_j = 2$ then the end-disc in question cannot be a point due to the extra condition on the cells $I_j$). Otherwise, say, $n_1 > 2$. Then there is an $A_i$ such that $A_i \sub I_1$. Pick any definable point in $A_i$, which splits $I_1$ into two open intervals and a point. Iterate this operation for all such cells $I_j$, and so on. Then $\max_j n_j$ decreases and eventually we reach an \LT-definable cell decomposition that is as desired. Note that if $(I_j)_j$ is an \omin-frame of $(A_i)_i$ then any \LT-definable refinement of the former is also an \omin-frame of the latter.

%We shall call these types by the obvious shorthand names lower-open, lower-closed, upper-open, and upper-closed.

\begin{rem}
Continuing the discussion above, we proceed to discuss the general case. Let $(I^n_k)_k$ be an \LT-definable cell decomposition of $\VF^n$. Let $(A^n_i)_i$ be an  $\lan{T}{RV}{}$-definable cell decomposition of $\VF^n$ that refines $(I^n_k)_k$ and is compatible with $A$ (that is, $A$ is the union of some of the cells $A^n_i$). Then there are finitely many $\lan{T}{}{}$-definable functions $f_j^{n-1} : \VF^{n-1} \fun \VF$, including
\begin{itemize}[leftmargin=*]
  \item characteristic functions of the cells $A^n_i$ in the last coordinate,
  \item (any) extensions of the last boundary functions of the cells $I^n_k$
\end{itemize}
such that, for each $a \in \VF^{n-1}$, the partition of $\VF$ induced by the points $f_j^{n-1}(a)$ is an \omin-frame of the cell decomposition $(A^n_{i,a})_i$ of $\VF$. This simply follows from the construction just described above and compactness.

Let $(I^{n-1}_k)_k$  be an \LT-definable cell decomposition  of $\VF^{n-1}$, compatible with the sets $\pr_{< n}(I^n_k)$, such that
\begin{itemize}[leftmargin=*]
  \item for all $j$, $f_j^{n-1} \rest I^{n-1}_k$ is continuous,
  \item for all $j$, $j'$ and all $a, a' \in I^{n-1}_k$, $f_j^{n-1}(a) < f_{j'}^{n-1}(a)$ if and only if $f_j^{n-1}(a') < f_{j'}^{n-1}(a')$.
\end{itemize}
Then the \LT-definable cell decomposition of $\VF^n$ induced by $(I^{n-1}_k)_k$ and the functions $f_j^{n-1}$ is a refinement of $(I^n_k)_k$. Let $(A^{n-1}_i)_i$ be a cell decomposition  of $\VF^{n-1}$, compatible with both $(I^{n-1}_k)_k$  and $\pr_{< n}(A^n_i)$. So $(A^{n-1}_i)_i$ induces, via $(A^n_i)_i$, a cell decomposition  of $\VF^{n}$. Repeating the procedure above with respect to $(A_i^{n-1})_i$ and $(I^{n-1}_k)_k$, we obtain  $\lan{T}{}{}$-definable functions $f_j^{n-2} : \VF^{n-2} \fun \VF$, and so on; in particular, each $f_j^{0}$ is actually just a point in $\VF$.
\end{rem}

The functions $f_j^i$, as boundary functions, determine  an \LT-definable cell decomposition of $\VF^n$. We concentrate on one of these cells, say $C$, and write, after renaming the functions if necessary, $(f_1^{i}, f_2^{i})_{C^i}$ for the intermediate cells, where $C^i = \pr_{\leq i}(C)$ and possibly $f_1^{i} \rest C^i = f_2^{i} \rest C^i$. The construction above also yields an $\lan{T}{RV}{}$-definable cell decomposition $(B_j)_j$ of $C$ such that
\begin{itemize}[leftmargin=*]
 \item the sets $B^{i}_j = \pr_{\leq i}(B_j)$ form a cell decomposition of $C^i$,
 \item for each $b \in B^{i}_j$, the induced cell decomposition of the interval $(f_1^{j}(b), f_2^{j}(b))$ is either trivial or bipartite.
\end{itemize}
Thus there are at most $2^n$ cells in this cell decomposition $(B_j)_j$ of $C$. Observe that, for each cell $B_j$, if  the type of its last coordinate involves open or closed end-discs then these end-discs must ne among the end-discs in the last coordinate of a single cell $A^n_i$, similarly for the other coordinates. This confirms that the sets $B_j$ are indeed cells, that is, they are equipped with characteristic functions. In more detail, starting with the last coordinate, if the end-discs in question are closed then  an   $\lan{T}{}{}$-definable function $g_j^{n-1} : B^{n-1}_j \fun \VF$ with $f_1^{n-1} < g_j^{n-1} < f_2^{n-1}$ can be constructed such that either $(f_1^{n-1}, g_j^{n-1})$ or $(f_2^{n-1}, g_j^{n-1})$ form a  characteristic function for both of the cells whose images under $\pr_{< n}$ are  $B^{n-1}_j$, and if the end-discs in question are open then the construction already guarantees that $f_1^{n-1}$, $f_2^{n-1}$  form such a characteristic function, similarly for the second last coordinate, and so on.

\begin{rem}
Using the same construction as in the proof of Lemma~\ref{vbox}, we see that  there are  $\lan{T}{}{}$-definable  bijections $\sigma_j : \VF^n \fun \VF^n$ and $\sigma_j^i : \VF^{i} \fun \VF^i$, $1 \leq i \leq n$,  such that
\begin{itemize}[leftmargin=*]
  \item $\sigma_j(C)$ is a box each one of whose sides is $0$ or $\VF^+$,
  \item $\sigma_j(B_j)$ is a $\vv$-box each one of whose sides is $0$ or $\VF^+$ or $\MM^+$ or $(\MM, +\infty)$ (the last two possibilities may be replaced by  $\OO^+$ and  $(\OO, +\infty)$, after applying the function $x \efun 1 /x$),
  \item $\sigma_j^i \circ \pr_{\leq i} = \pr_{\leq i} \circ \sigma_j$ and, for all $j$, $j'$, $i$, if $B_j^i = B_{j'}^i$ then $\sigma_j^{i+1} = \sigma_{j'}^{i+1}$.
\end{itemize}
However, we cannot find such a bijection that simultaneously turns every $B_j$ into a $\vv$-box, otherwise the theorem below would be much easier to prove.
\end{rem}

\begin{nota}\label{tree}
Let  $\bar B_j = \sigma_j(B_j)$ and $\bar B^i_j = \pr_{\leq i}(\bar B_j) = \sigma_j^i (\pr_{\leq i}(B_j))$. The set of these $\vv$-boxes gives rise to an obvious binary tree whose branches are of the form $(\bar B^1_j, \bar B^2_j, \ldots, \bar B^n_j)$. This tree is rooted if $B^1_j = C^1$ for some (hence all) $j$. By slight abuse of notation, we shall denote the path $(\bar B^1_j, \bar B^2_j, \ldots, \bar B^i_j)$ simply by its last vertex $\bar B^i_j$ when there is no danger of confusion. Given two paths $\bar B^i_{j}$ and $\bar B^{i'}_{j'}$, let $\bar B^i_{j} \sqcap \bar B^{i'}_{j'}$ denote the longest path they share. For each $j > 1$, let $l(j) < j$ be the least index such that the path $\bar B_{l(j)} \sqcap \bar B_{j}$ is the longest among all the paths $\bar B_{j'} \sqcap \bar B_{j}$, $j' < j$.

For each $i < n$, there is a similar function $l^i$ with respect to the paths $\bar B^i_j$; here we need to delete all the later repetitions, since possibly $\bar B^i_{j} = \bar B^{i}_{j'}$ for $j < j'$, but we will not reindex them for the following reason: if $\bar B^i_{j}$ is not a repetition then we have $l(j) = l^i(j)$.
\end{nota}

%We begin to examine the types of the cell decompositions of $C_b = (f_1^{n-1}(b), f_2^{n-1}(b))$, $b \in C^{n-1}$. Suppose that $C_b$ is of the type lower-open and $C_{b'}$ is of the type upper open. Then $\tfrac 1 2 C_b \coloneqq (f_1^{n-1}(b), \tfrac 1 2 f_2^{n-1}(b) + \tfrac 1 2 f_1^{n-1}(b))$ is still of the type   lower-open and yet $\tfrac 1 2 C_{b'}$ is of the trivial type. Similarly if $C_{b'}$ is of the type upper closed. Thus, replacing $f_2^{n-1}$ with $\tfrac 1 2 f_2^{n-1} + \tfrac 1 2 f_1^{n-1}$ if necessary, we may assume that no  $C_{b'}$ is of the type upper open or upper closed. Next, we may assume that $f_1^{n-1}$ is the constant function $0$ and $f_2^{n-1}$ is the constant function $1$, and hence if  $C_{b}$ is of the type upper open then the open disc in question is $\MM$ and  if  $C_{b}$ is of the type upper closed then the closed disc in question is $\OO_{\gamma}$ for some $\gamma > 0$.

We are now ready to prove the main theorem of this paper.

\begin{thm}\label{approx:ex}
Every $\lan{T}{RV}{}$-definable set $A \sub \VF^n$ admits a $T_{n-1}$-approximation.
\end{thm}

\begin{proof}
As in the proof of Lemma~\ref{LT:lim}, that every $P(b)$ is $b_\zeta$-$\lan{T}{}{}$-definable will be rather clear from the construction, so we will not be concerned with it. Also, by the discussion above and Lemma~\ref{omin:par}, without loss of generality, we may assume that   $A$ is contained in $C$. Then, by the construction of $(B_j)_j$, $A$ is a union of some of the cells $B_j \sub C$. In fact, each set $A^i = \pr_{\leq i}(A)$ is a union of some of the cells $B^i_j \sub C^i$.

The proof proceeds in four stages. In the first two stages we describe how the approximation is constructed. Then we show that the construction indeed yields an approximation.

\textbf{Stage (1)}: We begin with some preliminary preparations. Since $\bar B_j$ is a cell, we can construct a $T_{n-1}$-approximation $\bar P_j$ of it as in  the proof  of Lemma~\ref{LT:lim}. Observe that the function $P_{j}$ on $\VF^+$ given by $b \efun \sigma_j^{-1}(\bar P_{j}(b))$ is indeed a  $T_{n-1}$-approximation  of $B_j$. However, these $\bar P_j$ are not good enough for our purpose. We will carry out a more careful construction by induction on the number of cells $B_j$, in Stage (2), so to make the resulting approximations satisfy some extra properties. There are a few general modifications that are applied in each step of this inductive construction, which we list here and will not mention them again below.

\begin{enumerate}[leftmargin=*]
  \item A cell is defined with respect to a specific order of the coordinates and hence the construction in the proof  of Lemma~\ref{LT:lim} also relies on this order. But $\bar B_j$ is a very simple cell, namely a $\vv$-box, so  we can  construct a $T_{n-1}$-approximation with respect to any order of the coordinates, in particular, we can start with the $n$th coordinate and move downward, and this is how each $\bar P_j(b)$, which is indeed a box, is constructed.
  \item In each step of the construction of each $\bar P_j$, if the $\vv$-interval in question is $\MM^+$ then  half-closed intervals of the form $(0, b]$ are used in the corresponding approximation, instead of open intervals, as in (\ref{infinter}).
  \item The constructions of different approximations $\bar P_j$ may employ different growth representatives (closely related to $(\zeta_i)_i$, but $(\zeta_i)_i$ itself is quantified over, that is, it is in effect a variable).
\end{enumerate}

Eventually, when stitching these approximations $P_j$ together to form a single function $P$ on $\VF^+$, which is meant to be a $T_{n-1}$-approximation of $A$, we wish to have certain flexibility in choosing which fiber of $P_j$ the set $P(b)$ actually contains; to be more precise, instead of simply setting $P(b) = \bigcup_j P_j(b)$ for every $b \in \VF^+$, we can choose a sequence of strictly increasing $\lan{T_{n-1}}{}{}$-definable functions $s_j : \VF^+ \fun \VF^+$ and then  set $P(b) = \bigcup_j P_j (s_j(b))$. For a single $j$, the functions $P_j$ and $P_j \circ s_j$ may seem more or less  the same, in particular, both are $T_{n-1}$-approximations  of $B_j$. But for $j \neq j'$, the ``contact'' relation between $P_j(b)$ and $P_{j'}(b)$ is potentially different from that between $P_j (s_j(b))$ and $P_{j'} ( s_{j'}(b))$. This point will become clearer in the discussion below.   These functions $s_j$ will also be constructed inductively.

Note that the cells $P_j (s_j(b))$ do not form a partition of $P(b)$, since they may overlap.

\textbf{Stage (2)}: We now proceed to the inductive construction of $\bar P_j$ and $s_j$. For $0 \leq i \leq n$, let $\bar P^i_{j}$ be the function given by $b \efun \pr_{\leq i}(\bar P_{j} (b) )$. The following condition needs to be satisfied in the inductive step:
\begin{quote}
for every $j' < j$ and every $0 \leq i \leq n$, if $\bar B^i_{j'} = \bar B^i_{j}$ then $\bar P^i_{j'} \circ s_{j'} = \bar P^i_{j} \circ s_{j}$, and if, in addition, $\bar B^{i+1}_{j'} \neq \bar B^{i+1}_{j}$ then, for all $b \in \VF^+$,
\[
\bar P^{i+1}_{j'} ( s_{j'}(b))  \cup \bar P^{i+1}_{j} (s_{j}(b)) = \VF^+ \times \bar P^{i}_{j} ( s_{j}(b)).
\]
\end{quote}
Alternatively, we can simply require that the condition holds for $l(j)$ instead of every $j' < j$; it is not difficult to see that this is an equivalent formulation.

For the base case $j=1$, simply construct $\bar P_1$ as in the proof  of Lemma~\ref{LT:lim}, subject to the general modifications described above, and take $s_1 = \id$. For the inductive step, suppose that we  have constructed $\bar P_{j'}$ and $s_{j'}$ for all $j' < j$. To construct $\bar P_{j}$ and $s_{j}$, we only need to concentrate on $\bar B_{l(j)}$. For simplicity, write $l(j)$ as $l$. Denote the length of $\bar B_{l} \sqcap \bar B_{j}$ by $k$. Each $\bar P_{l} ( s_l(b))$ is a box whose sides are of the form $(0, d_i(b)]$ or $(d_i(b), +\infty)$ or $0$ or $(0, +\infty)$, corresponding to the sides $\MM^+$ or $(\MM, +\infty)$ or $0$ or $(0, +\infty)$ of the $\vv$-box $\bar B_j$; for the last two cases we set $d_i(b) = 0$. The sequence of points $d_n(b), \ldots, d_1(b)$, where the indices indicate the corresponding coordinates (remember that the construction  of  $\bar P_{l}$ starts in the $n$th coordinate), may be written as
\[
\dot \zeta_{l,0}\circ s_l(b), \ldots, \dot \zeta_{l,(n-1)} \circ s_l(b),
\]
where $\zeta_{l, i} : \VF^+ \fun \VF^+$ is a $T_{i}$-infinite function and $\dot \zeta_{l,i}$ is the function $1 / x \circ \zeta_{l,i} \circ 1/x$ or the function $1/x \circ \zeta^{-1}_{l,i} \circ 1/x$ or the constant function $0$, corresponding to the three possibilities of $d_i(b)$ as well as the three cases indicated towards the end of  the proof of Lemma~\ref{LT:lim}.
\begin{itemize}[leftmargin=*]
  \item For $0 \leq i \leq n-k-1$, let $\dot \zeta_{j,i} : \VF^+ \fun \VF^+$ be a function of one of these three forms, with $\zeta_{j,i} = \zeta_{l,i}$, such that the boxes determined by the sequence of points $\dot \zeta_{j,0}(b), \ldots, \dot \zeta_{j,(n-k-1)}(b)$, $b \in \VF^+$, form a $T_{n-k-1}$-approximation of $\pr_{> k}(\bar B_j)$. Note that, by the definition of $k$, we have $\dot \zeta_{j,(n-k-1)} = \dot \zeta_{l,(n-k-1)}^{-1}$, which exists, at least near $0$, since $\dot \zeta_{l,(n-k-1)}$ cannot be $0$.
  \item For $n-k-1 < i \leq n-1$, let
\[
\dot \zeta_{j,i} = \dot \zeta_{l,i} \circ \dot \zeta_{l,(n-k-1)}^{-1} \circ \dot \zeta_{l,(n-k-1)}^{-1}.
\]
Observe that the functions $\dot \zeta_{j,i}$, $\dot \zeta_{l,i}$ are of the same ``growth'' type near $0$, that is, both are $T_i$-infinitesimal or infinitesimal-$T_i$-constant or $0$.
\end{itemize}
Therefore, the boxes $\bar P_j(b)$, $b \in \VF^+$, determined by the sequences of points $\dot \zeta_{j,0}(b), \ldots, \dot \zeta_{j,(n-1)}(b)$ form a $T_{n-1}$-approximation $\bar P_j$ of $\bar B_j$. Set
\[
s_j = \dot \zeta_{l,(n-k-1)} \circ \dot \zeta_{l,(n-k-1)} \circ s_l.
\]
Then $\bar P_j$, $s_j$ are as required. Note that, since $s_j : \VF^+ \fun \VF^+$ is a strictly increasing $\lan{T_{n-1}}{}{}$-definable function, it follows that $s_j(u) \ssin \UU^+$ for all $u \ssin \UU^+$ and hence $\bar P_j \circ s_j$ is  also a $T_{n-1}$-approximation of $\bar B_j$.

%Since $A$ is not $\lan{T}{}{}$-definable, some $\dot \zeta_i$ is not $0$; we assume that, for all $\hat P_{j}$, all of them are not $0$, as other cases can be handled in a similar way (in a sense they are all ``degenerate'' versions of this case).

\textbf{Stage (3)}: Let $P_j$ be as defined in Stage~(1). For ease of notation, we will just write $\bar P_j \circ s_j$, $P_j \circ s_j$ as $\bar P_j$, $P_j$; this  will not cause confusion  since $\bar P_j$, $P_j$, and $s_j$ will always occur in these composite forms. Let $P = \bigcup_j P_j (b)$. We claim that  $P$ is a $T_{n-1}$-approximation  of $A$. To show this, we proceed by induction on $n$. Let us begin by clarifying a bit what the inductive hypothesis is.
\begin{enumerate}[leftmargin=*]
  \item The claim holds for all $\lan{T}{RV}{}$-definable sets $A$ and all cell decompositions of $A$ of the special form described above, in particular, for the set $A^i$ and the cells $(B^i_j)_j$.
  \item The claim holds for all hypogenous theories $T$ and all power progressions at $T$, in particular, for any subsequence $(T_i, T_{i+1}, \ldots)$ of $(T_i)_i$, since it is a power progression at the hypogenous theory $T_i$.
  \item Finally, as we have mentioned earlier, there is also a quantification over all growth representatives for a power progression, in particular, for any subsequence  $(\zeta_{i}, \zeta_{i+1}, \ldots)$ of $(\zeta_{i})_i$, which is a growth representative for $(T_i, T_{i+1}, \ldots)$.
\end{enumerate}
%(this is a representative sequence because, for $i> 0$, $\zeta_{i} \circ \zeta^{-1}_{0}$ is also $T_{i-1}$-infinite).

From here on fix any $\usu$. The base case $n=1$ is essentially Example~\ref{exam:approx}. Note that  here we already need the second  modification in the construction of $\bar P_j$ described in Stage~(1) if $A$ is a union of two (nonempty) cells, for otherwise there will be a ``gap'' between $\bar P_1(u)$ and $\bar P_2(u)$, namely $u$ itself.

For the inductive step, since $P_j$ is a $T_{n-1}$-approximation of $B_j$, it is clear that if $V$ is an $\lan{T}{}{}$-definable set containing $A$ then $P(u) \sub V$. The more difficult task is to show that if $V\sub A$ then $V \sub P(u)$. To that end, without loss of generality, we may assume that $V$ is a cell. We claim that the functions $P^{n-1}_j : b \efun \pr_{<n}(P_j(b))$ indeed form a $T_{n-1}$-approximation $P^{n-1} : b \efun \bigcup_j P^{n-1}_j(b)$ of $A^{n-1}$ (note that this is not a $T_{n-2}$-approximation). To see this, observe that the box $\pr_{<n}(\bar P_{j}(b))$ is determined by the  points
\[
(\dot \zeta_{j,1} \circ s_j) (b), \ldots, (\dot \zeta_{j,n-1} \circ s_j) (b),
\]
and hence the sequence of functions $b \efun \pr_{<n}(\bar P_{j}(b))$ can be constructed by applying the inductive procedure in Stage~(2)  to the $\vv$-boxes $\bar B^{n-1}_j$, using instead the growth representative $(\zeta_1, \zeta_2, \ldots)$ and the function $l^{n-1}$ defined in Notation~\ref{tree} (more formally, in light of the last sentence of Notation~\ref{tree}, this can be verified through a simple induction). Since $\sigma_j^{n-1}(B_j^{n-1}) = \bar B_j^{n-1}$,  the inductive hypothesis  yields that $P^{n-1}$  is a $T_{n-1}$-approximation of $A^{n-1}$ and hence $V^{n-1} \coloneqq \pr_{<n}(V) \sub P^{n-1}(u)$.

%insofar as the set $\hat p^{n-1}_j  = \bigcup_{b \in \VF^+} \pr_{<n}(\hat P_{j}(b)) \times b$ is concerned, the sequence of points $\dot \zeta_{j1} \circ s_j (b), \ldots, b$ that determine the box $\pr_{<n}(\hat P_{j}(b))$ may be equivalently written as
%\[
%s'_j(b), (\dot \zeta_{j2} \circ \dot \zeta^{-1}_{j1}) \circ s'_j(b), \ldots, b,
%\]
%where $s'_j = \dot \zeta_{j1} \circ s_j$. Each function $1 / x \circ \dot \zeta_{ji} \circ \dot \zeta^{-1}_{j1} \circ 1/x$ is of the same ``growth'' type as the function $1 / x \circ \dot \zeta_{ji} \circ 1/x$ near infinity, by which we mean that they are either both $T_{i-1}$-infinite or both infinite-$T_{i-1}$-constant, because $i >1$. It follows from the proof  of Lemma~\ref{LT:lim}  that $\hat p^{n-1}_j$ is indeed a $T_{n-1}$-approximation of the $\vv$-box $\pr_{<n}(\sigma_j(B_j))$.

We concentrate on one cell  $B^{n-1} \sub A^{n-1}$ with $B^{n-1} \cap V^{n-1} \neq \0$. There are at most two cells $B_j$ with $B_j \sub A$ and $B_j^{n-1} = B^{n-1}$. After re-enumeration if necessary, we may assume that these cells are $B_1$ and, if it exists, $B_2$.  Let
\[
V^{n-1}(b)  = P^{n-1}_1(b) \cap V^{n-1}, \quad \text{for } b \in \VF^+.
\]
By Lemma~\ref{in:out:same}, $V^{n-1}(u) \neq \0$ is equivalent to $B^{n-1} \cap V^{n-1} \neq \0$. It is enough to show that, for all $a \in V^{n-1}(u)$, $V_a \sub P_1(u)_a \cup P_2(u)_a$. To that end, observe that we can actually apply $\sigma_1$ (which equals $\sigma_2$ if $\sigma_2$ exists) and show this   property for $\sigma_1(A)$, $\sigma_1(V)$, etc. So we may assume that $\sigma_1 = \id$, in particular, $B_1$, $B_2$ are $\vv$-boxes whose last coordinates are either $\MM^+$ or $(\MM, +\infty)$, and $P_1(b)$, $P_2(b)$ are boxes for all $b \in \VF^+$. There are three cases to consider.

\textbf{Stage (4)}: The first case is that  $B_2$ does not exist and the last coordinate $\pr_n(B_1)$ of $B_{1}$ is $\VF^+$. Then  the property is clear since $V$ is certainly contained in $C$ and the last coordinate of $P_1(u)$ is also $\VF^+$.

Similarly, if $B_2$ does exist then, by the construction, the last coordinates of $P_1(u)$, $P_2(u)$ are the intervals $(0, (\dot \zeta_{1,0} \circ s_1)(u)]$, $((\dot \zeta_{1,0} \circ s_1)(u), +\infty)$, and hence the property also holds in this case. Note that here $s_1$ is not necessarily $\id$, since we have re-enumerated $(P_j)_j$ so as to simplify the notation. Actually this proof seems to be more complicated than it ought to be primarily because of this case. More specifically, this case is  the reason why the approximations $\bar P_j$ are constructed backwards along the coordinates instead of starting with the first coordinate as in the proof of Lemma~\ref{LT:lim}, for otherwise the opposing ends of the intervals in the last coordinates of $P_1(u)$, $P_2(u)$ are defined using different functions, one $T_{n-1}$-infinite and the other infinite-$T_{n-1}$-constant, which would leave a gap between the two intervals.

The third case is that  $B_2$ does not exist but $\pr_n(B_1)$ is not $\VF^+$. There are two subcases to consider. First suppose that $\pr_n(B_1)$ is $(\MM, + \infty)$. If the $i$th coordinate  of $B^{n-1}$ is   $(\MM, +\infty)$ then the $i$th coordinate of $P^{n-1}_1(u)$ is an open interval $I_u$ contained in   $(\MM, +\infty)$, otherwise it is a half-closed interval $H_u \sub \VF^+$ containing $\MM^+$. For  $e \in \VF^+$, we replace, in $P^{n-1}_1(u)$, the intervals $I_u$ with $(u, +\infty)$ and  the intervals $H_u$ with $(0, e)$, and call the resulting box $Q(u, e)$, which is contained in $B^{n-1}$ if and only if $e \in \MM^+$.  Let
\[
V^{n-1}(u, e)  = Q(u, e) \cap V^{n-1},
\]
which cannot be empty if $e \in (\MM, +\infty)$, because $\usu$ and $B^{n-1} \cap V^{n-1} \neq \0$. Then there is a $u$-$\lan{T}{}{}$-definable function $\rho_u: \VF^+ \fun \VF^+_{0} \cup \{+\infty\}$ such that $\rho_u(e)$ is the largest element satisfying the condition
\begin{equation}\label{u:condish}
 V(u,e) \coloneqq \bigcup_{a \in V^{n-1}(u, e)} V_a   \sub [2\rho_u(e), +\infty) \times Q(u, e),
\end{equation}
where $\rho_u(e) = +\infty$ if and only if $V^{n-1}(u, e)$ is empty. Here the factor 2 is just a simple device to guarantee $V(u,e) \sub (\rho_u(e), +\infty) \times Q(u, e)$, which is what we actually need. So if $e \in \MM^+$ then $\rho_u(e) > \MM$. Note that $\rho_u$ is not increasing, and hence there is a $u$-definable $e_u \in \UU^+$ such that
\[
+\infty > \rho_u(e_u) > \MM \dand V(u,e_u) \sub  (\rho_u(e_u), +\infty) \times Q(u, e_u).
\]
Near $\MM$, we may speak of the $\lan{T}{RV}{}$-definable functions $\tau: \UU^+ \fun \UU^+$ given by $u \efun e_u$ and $\rho : \UU^+ \fun (\MM, +\infty)$ given by $u \efun \rho_u(e_u)$. Suppose that  $\rho$ admits an inverse  $\rho^{-1}: \UU^+ \fun \UU^+$ near $\MM$. Since
\begin{itemize}[leftmargin=*]
  \item $(\dot \zeta_{1,(n-i)} \circ s_1)(u) > (\rho^{-1} \circ \dot \zeta_{1,0} \circ s_1)(u)$ if the $i$th coordinate of $B^{n-1}$ is $(\MM, +\infty)$ and
  \item $(\dot \zeta_{1,(n-i)} \circ s_1)(u) < (\tau \circ \rho^{-1} \circ \dot \zeta_{1,0} \circ s_1)(u)$  if the $i$th coordinate of $B^{n-1}$ is $\MM^+$,
\end{itemize}
it follows that the box $P^{n-1}_1(u)$ is contained in the box
\[
Q((\rho^{-1} \circ \dot \zeta_{1,0} \circ s_1)(u), (\tau \circ \rho^{-1} \circ \dot \zeta_{1,0} \circ s_1)(u)).
\]
Since $(\rho \circ \rho^{-1} \circ \dot \zeta_{1,0} \circ s_1)(u) = (\dot \zeta_{1,0} \circ s_1)(u)$, we see that $V_a   \sub ((\dot \zeta_{1,0} \circ s_1)(u), +\infty)$ for all $a \in V^{n-1}(u)$ as desired. Now, if $\rho$ does not admit an inverse near $\MM$, that is, if there is a definable $e^* \in \UU^+$ such that $\rho(u) > e^*$ for all $\usu$,  then we can simply replace $\rho$ with an  $\lan{T}{RV}{}$-definable function $\UU^+ \fun \UU^+$ that does admit an inverse near $\MM$, and since lowering  $\rho$ is conducive to our purpose, we are back in the situation above. Note that if no coordinate  of $B^{n-1}$ is $\MM^+$ then the entire construction is simpler since the function $\tau$ will not even appear.

Finally, the subcase that $\pr_n(B_1)$ is $\MM^+$ is similar. We just point out the necessary modifications. The function $\rho_u$ is so defined that $\rho_u(e)$ is the least that satisfies the condition (\ref{u:condish}), where of course there is no need for the factor $2$ anymore and $[2\rho_u(e), +\infty)$ is simply replaced by $(0, \rho_u(e)]$, and $\rho_u(e) = 0$ if and only if $V^{n-1}(u, e)$ is empty. So $\rho_u$ is not decreasing and if $e \in \MM^+$ then $\rho_u(e) \in \MM^+_0$ as well. Since $\rho_u$ is $u$-$\lan{T}{}{}$-definable, there is an $\lan{T}{RV}{}$-definable function $\tau : \UU^+ \fun \UU^+$ near $\MM$ such that $\rho(u) \coloneqq \rho_u(\tau(u)) \in \OO^+$. So either  $\rho(u) \in \MM^+$ for all $\usu$ or $\rho(u) \in \UU^+$ for all $\usu$. It is not hard  to see that if the second possibility occurs then $\tau$ may be adjusted so that the function $\rho$ admits an inverse near $\MM$; in fact, the construction above suggests that raising $\rho$ near $\MM$ is conducive to our purpose and hence we may assume that this condition on $\rho$  holds in either case. At this point, the same argument goes through.
\end{proof}

\begin{rem}
If $A' \sub A$ is another $\lan{T}{RV}{}$-definable set then we can find  $T_{n-1}$-approximations $P$, $Q$ of $A$, $A'$, respectively, with  $Q(u) \sub P(u)$. To see this, recall that we started with an arbitrary cell decomposition $(A^n_i)_i$ of $A$ and then proceeded to construct this  cell decomposition $(B_j)_j$. We can certainly make all this compatible with $A'$, that is, $A'$ is also a union of some of these cells $B_j$. The proof of Theorem~\ref{approx:ex} constructs inductively an $T_{n-1}$-approximation of each $B_j$ and also specifies a way to stitch them together to form a $T_{n-1}$-approximation of $A$. This procedure of course depends on an enumeration of $(B_j)_j$ (see Remark~\ref{tree}). But then we can enumerate those cells $B_j$ contained in $A'$ first. Consequently, the construction  yields  $T_{n-1}$-approximations $P$, $Q$ of $A$, $A'$ as desired, more precisely, $P(b) = \bigcup_j P_j (s_j(b))$ and $Q(b) = \bigcup_{1 \leq j \leq k} P_j (s_j(b))$ for some $k$.
\end{rem}

\providecommand{\bysame}{\leavevmode\hbox to3em{\hrulefill}\thinspace}
\providecommand{\MR}{\relax\ifhmode\unskip\space\fi MR }
\providecommand{\MRhref}[2]{%
  \href{http://www.ams.org/mathscinet-getitem?mr=#1}{#2}
}
\providecommand{\href}[2]{#2}

%
%%---------------------------------------------------------------------
%%Included for Gather Purpose only:
%%input "C:\localtexmf\bibtex\bib\mybib\MYbib.bib"
%\bibliographystyle{amsplain}
%\bibliography{../../mybib/MYbib}
%%---------------------------------------------------------------------

\end{document}